\documentclass[12pt]{amsart}

\usepackage[hmargin=2.5cm,bmargin=2.5cm,tmargin=3cm]{geometry}
\usepackage{graphicx}
\usepackage{color}

\usepackage[english]{babel}
\usepackage{amsmath}
\usepackage{amssymb}
\usepackage{algorithmicx}
\usepackage[ruled]{algorithm}
\usepackage{algpseudocode}
\usepackage{parskip}
\algnewcommand{\And}{\textbf{and}}

\newtheorem{thm}{Theorem}[section]
\newtheorem{lem}[thm]{Lemma}

\theoremstyle{definition}

\newtheorem{defn}{Definition}
\newtheorem{exmp}{Example}
\newtheorem{remark}[thm]{Remark}

\newcommand{\be}{\begin{equation}}
\newcommand{\ee}{\end{equation}}
\newcommand{\bea}{\begin{eqnarray}}
\newcommand{\eea}{\end{eqnarray}}

\newcommand{\Graph}{\mathcal{G}}
\newcommand{\Paths}{\mathcal{P}} 
\newcommand{\Path}{\mathcal{P}}  
\newcommand{\RPath}{\mathcal{R}}
\newcommand{\weight}{\mathcal{W}}
\newcommand{\NetworkGraph}{(\Graph, V^B, \Paths)}

\DeclareMathOperator{\PCD}{PCD}

\newcommand{\R}{\mathbb{R}}

\title{Graph Reconstruction from Path Correlation Data}
\thanks{This material is based upon work supported by DARPA and by the
  National Science Foundation under Grant DMS-1410657. The views, opinions and/or findings expressed are those of the authors and should not be interpreted as representing the official views or policies of the Department of Defense or the U.S. Government.}

\author[G.~Berkolaiko]{Gregory Berkolaiko}
\address{Department of Mathematics, Texas A\&M University, College
  Station, TX 77843-3368}
\author[N.~Duffield]{Nick Duffield}
\address{Department of Electrical and Computer Engineering, Texas A\&M University, College
  Station, TX 77843}
\author[M.~Ettehad]{Mahmood Ettehad}
\address{Department of Mathematics, Texas A\&M University, College
  Station, TX 77843-3368}
\author[K.~Manousakis]{Kyriakos Manousakis}
\address{Perspecta Labs, Basking Ridge, NJ 07920}

\begin{document}
\maketitle

\begin{abstract}
A communication network can be modeled as a directed connected graph
with edge weights that characterize performance metrics such as loss
and delay.  Network tomography aims to infer these edge weights from
their pathwise versions measured on a set of intersecting paths
between a subset of boundary vertices, and even the underlying graph
when this is not known. In particular, temporal correlations between path metrics have been used infer composite weights on the subpath formed by the path intersection. We call these subpath weights the Path Correlation Data.

In this paper we ask the following question: when can the underlying
weighted graph be recovered knowing only the boundary vertices and the
Path Correlation Data? We establish necessary and sufficient
conditions for a graph to be reconstructible from this information,
and describe an algorithm to perform the reconstruction. Subject to
our conditions, the result applies to directed graphs with asymmetric
edge weights, and accommodates paths arising from asymmetric routing
in the underlying communication network. We also describe the
relationship between the graph produced by our algorithm and the true
graph in the case that our conditions are not satisfied.  
\end{abstract}



\section{Introduction}
\subsection{Background and Motivation}

\newcommand{\paraspace}{\vspace{0.02in}}
\newcommand{\parab}[1]{\paraspace\noindent{\textsl{#1}}}
\def\cG{\mathcal G}
\def\cW{\mathcal W}
\def\cD{\mathcal D}
\def\cP{\mathcal P}

\parab{Graphs and Communication Networks.}
The problem that we study originates in performance measurement of
packet communications networks. These are modeled as a a weighted
directed graph $\cG=(V,E,\cW)$ in which the vertex set $V$ represents
routers, the edges $E$ represent directed links between routers, and
the edge weights $\cW$ characterize performance metrics of the
associated links.  A subset $V^B\subset V$ of vertices represents the
``boundary'' of the network where the measurements are performed.  We
assume that there is a fixed directed path $\Path(b,b')$ connecting each
boundary node pair $b,b' \in V^B$ and that one can measure the common
portion of any two paths.

More specifically, we will assume one can measure the length of the
following: the path between any two boundary nodes, the common part of
the two paths from a boundary node to any two other boundary nodes, as
well as the common part of the paths to a boundary node from any
two other boundary nodes.  This set of measurements we will call the
Path Correlation Data (PCD).  We establish necessary and sufficient
conditions --- which turn out to be rather natural --- for the
reconstruction of a graph from its PCD and present an algorithm to
achieve it.  In the case when the underlying graph violates the
reconstructibility conditions, we describe the result of the algorithm
--- it turns out to be the ``simplest'' routing network that produces
the observed PCD.

\parab{Network Tomography and the Inversion Problem.}
The form of our model and the assumptions we make originate from a
body of work developed under the term \textsl{Network Tomography}
\cite{10.2307/2291416} that seeks to infer link metrics and even the
underlying network topology from the measured metric values on paths
traversing the network between a set of routers at the network
boundary, represented by $V^B$.  This setting is similar to other
graph reconstruction problems, such as tomography of electrical
resistance networks (see, e.g., \cite{Cheney99,MR2719770}), optical
networks \cite{MR3634446}, and graph reconstruction from
Schr\"odinger-type spectral data (see, e.g.,
\cite{MR2353313,MR2545980}).  However, in a communication network
model there is a single path between given origin and destination, in
contrast to the electrical current flowing between two points a
resistive medium via all possible paths.  In this sense, our model
is more similar to combinatorial reconstruction problems
\cite{MR1047783,EvaLan_aam18}.

In many practical cases the communication network metrics are additive
in the sense that the sum of metric values over links in a path
corresponds to the same performance metric for the path. Examples of
additive metrics include mean packet delay, log packet transmission
probability, and variances in an independent link metric model.  For
additive metrics, a putative solution to the network tomography
problem attempts to invert a linear relation expression between the
set of path metrics $\mathcal{D}$ and the link metrics $\mathcal{W}$ in the form
\begin{equation}
  \label{eq:lin} 
  \mathcal{D} = \mathcal{A} \mathcal{W}
\end{equation}
Here $\mathcal{A}$ is the incidence matrix of links over paths,
$\mathcal{A}_{\mathcal{P},\ell}$ is 1 if path $\mathcal{P}$ traverses link
$\ell$, and zero otherwise. The linear system (\ref{eq:lin}) is
generally underconstrained in real-life networks, and hence does not
admit a unique solution \cite{4016134}. To overcome this deficiency,
one approach has been to impose conditions on the possible solutions,
typically through sparseness, effectively to find the ``simplest''
explanation of the observed path metrics; see
\cite{DBLP:journals/tit/Duffield06,10.1007/978-3-642-30045-5_22}. A
different approach in the similar problem of traffic matrix tomography
has been to reinterpret (\ref{eq:lin}) as applying to bi-modal
measurements of packet and bytes counts \cite{SinMic_ip07}, or 
of empirical means and variances then imposing constraints between these based on
empirical models \cite{10.2307/2291416,866369}. However, the high
computational complexity of this approach makes it infeasible for
real-world communications networks \cite{Zhang:2003:FAC:781027.781053}, although
quasi-likelihood methods offer some reduction in
complexity \cite{1212664}.  A related approach known as Network
Kriging seeks to reduce dimensionality in the path set by assumption
on prior covariances \cite{4016134}.

\parab{Correlations and Trees.}
More relevant for the work of this paper has been the idea to 
exploit correlations between metrics on different paths that occur due
to common experience of packets on their intersection. One variant
uses multicast probing \cite{796384} or emulations thereof
\cite{DBLP:journals/ton/DuffieldPPT06}. Another variant exploits the
fact that variances of some measurable packet statistics are both
additive and independent over links.  If $\cD$ is such a
variance-based metric, then
$ \text{Cov}(\cD_{P_1}, \mathcal{D}_{P_2}) =
\text{Var}(\cD_{P_1\cap P_2})$
where $\cD_P$ denotes the metric value across path $P$; see
\cite{DBLP:journals/ton/DuffieldP04}.

We abstract both these cases into a unified data model in which for
every triple $b$, $b_1$ and $b_2$ of boundary vertices, we can
measure, via covariances of the packet statistics, the metric of the
intersection $P=\Path(b,b_1)\cap \Path(b,b_2)$ as well as to the metric of
the intersection $P=\Path(b_1,b)\cap \Path(b_2,b)$.  The results of such
measurements we will denote by $\PCD(b \prec b_1,b_2)$ and
$\PCD(b_1,b_2 \succ b)$ correspondingly. We remark that we use the
symbols $\prec$ and $\succ$ here not as binary comparison operators
but as pictograms meant to evoke the topological structure of the
corresponding pair of paths.  The totality of such measurements we
call the \textbf{path correlation data}.  Note that we will not, in
general, assume that the paths are symmetric; the path $\Path(b,b_1)$
may be different topologically from the path $\Path(b_1,b)$.  As a
consequence, the values $\PCD(b \prec b_1,b_2)$ and
$\PCD(b_1,b_2 \succ b)$ are in general different.  We will assume that
the function $\PCD$ is measured exactly; see below for a brief
discussion of possible sources of error and the techniques for
error-correction. 

Under fairly general conditions that will be in force in this paper,
the problem has a natural formulation in terms of trees. First assume,
that for each $b,b_1,b_2\in V^B$, the intersection
$\Path(b,b_1)\cap \Path(b,b_2)$ is connected. Second, assume that the
metric $\cD$ is path increasing, i.e., $\cD(\Path)<\cD(\Path')$ for
$\Path\subsetneq \Path'$. As shown in \cite{971737}, the quantities
$\{\PCD(b \prec b_1,b_2): b_1, b_2\in V^B\}$ give rise to an embedded
logical weighted tree rooted at $b$ (called ``source tree''). The tree
is computed iteratively by finding node pairs $(b_1,b_2)$ of maximal
$\PCD(b \prec b_1,b_2)$ and identifying each such pair with a branch
point in the logical tree. Each logical link is assigned a weight
equal to the difference between the values of
$\text{PCD}(b \prec \cdot,\cdot)$ associated with its end points.
Similarly, the quantities $\PCD(b_1,b_2 \succ b)$ give rise to an
embedded logical tree with a single receiver $b$ and the sources
$V^B \setminus \{b\}$.  Such trees are called ``receiver trees''.
Under the assumed conditions, our problem can be restated as how to
recover the underlying weighted graph from the set of logical source
and receiver trees rooted at every $b\in V^B$.

\parab{Summary of the Results}. The main contribution of this paper,
Theorem~\ref{thm:main}, is to show that under natural conditions, a
weighted directed graph $\cG$ can be recovered knowing only the
graph's Path Correlation Data (PCD).  The conditions under which
Theorem~\ref{thm:main} holds are (i) each edge is traversed by at
least one path in $\Path$; (ii) each non-boundary node is
\textsl{nontrivial} in the sense that in-degree and out-degree are not
both equal $1$; and (iii) each non-boundary node $x$ is
\textsl{nonseparable} in the sense that the set of paths
$\Path(x)\subset \Path$ that pass through $x$ cannot be partitioned
into two or more subsets with non intersecting end point sets. Our
result holds without the assumption of weight symmetry (defined as
requiring the existence of the reverse of any edge in $G$, having the
same weight) or path symmetry (defined as the paths in either
direction between two boundary nodes traversing the same set of
edges). We prove the correctness of our reconstruction algorithm
(Algorithm~\ref{algorithm:non_symmetric}) under the stated
assumptions.

Our solution does not assume that link weights are symmetric. Neither
do we assume that that paths in either direction between two endpoints
are symmetric: they are not required to traverse the same set of
internal (i.e. non-boundary) vertices. This level of generality
reflects networking practice, in which non-symmetric routing is
employed for policy reasons including performance and revenue
optimization \cite{Teixeira:2004:DHR:1012888.1005723}.  However, in
the cases where symmetric paths can be assumed {\it a priori}, this
knowledge enlarges the set of reconstructible networks.  We therefore
pay special attention to this case, providing alternative definitions
of the \emph{nontrivial} and \emph{nonseparable} vertices and a
separate proof of the correctness of
Algorithm~\ref{algorithm:non_symmetric} (which requires a one-line
change).  We also establish the correctness of a second, more
customized Algorithm~\ref{algorithm:symmetric} which applies only in
the case of symmetric weights and paths, and which is computationally
less expensive than Algorithm~\ref{algorithm:non_symmetric} applied to
this case.

\parab{Inference from Inconsistent Data.}
Before describing our approach in more detail, we note that practical
network measurement may provision data with imperfect consistency. For
example, input data may be provided in the form of weighted trees
computed from packet measurement over time intervals that are not
perfectly aligned, so that the metric of a path $\Path(b,b')$ may be
reported differently in the source tree from $b$ and the receiver tree
to $b'$. Even with aligned intervals, deviations from the model and
statistical fluctuations due to finitely many probe packets may result
in inconsistency. To enable our proposed algorithm to operate with
such data, we propose to compute a $\text{PCD}$ that is a
least-squares fit to inconsistent tree data. This extension and
associated error sensitivity analysis is described in the forthcoming
companion paper \cite{PrepUs2018}.

\subsection{An example of a communication network}

To illustrate the information available to an observer in our model,
consider the network graph schematically shown in
Fig.~\ref{fig:AltSource}. It is assumed that the end-to-end
measurements are possible among the boundary vertices
$V^B = \{b_1, ... , b_6\}$.  The three versions of the same graph
shown contain information about the paths between the given source
($b_1$, $b_2$ and $b_3$, correspondingly) and the
corresponding set of receivers $V^B \setminus \{b_i\}$; the links belonging to
these paths are highlighted in thicker lines.

\begin{figure}[t]
  \centering
  \includegraphics[scale=0.7]{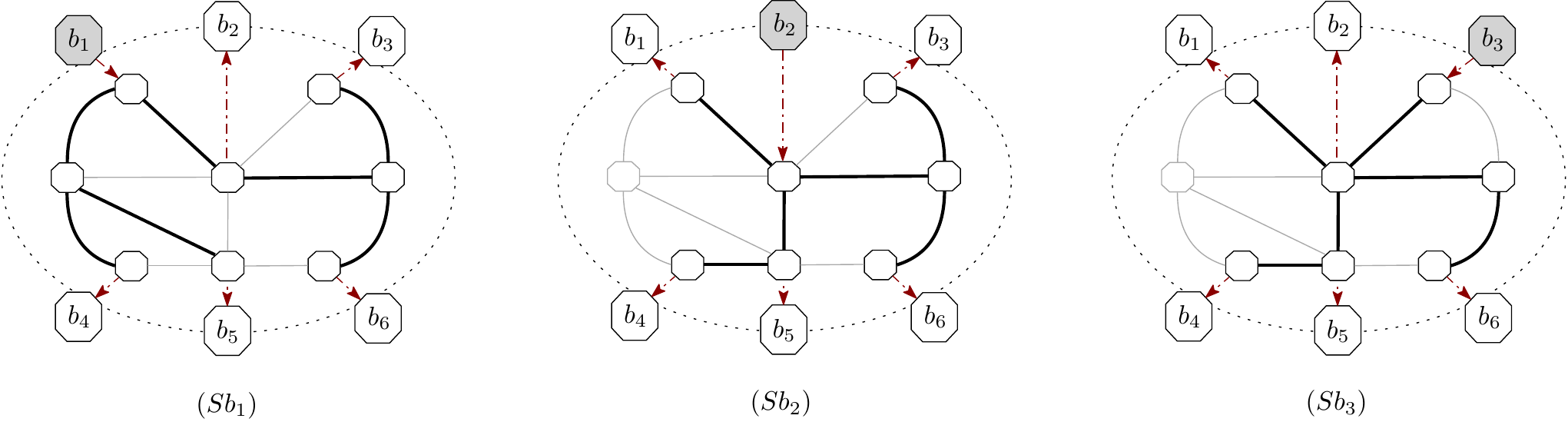}
  \caption{Alternative selections of the source and the corresponding
    routing paths (bold edges).} 
  \label{fig:AltSource}
\end{figure}

From the point of view of an external observer, the routing paths on
the graph are hidden but can be reconstructed, to a certain extent, by
the measurements with a fixed source and alternating receivers,
represented in our setting by queries to the $\PCD$ function. As
Fig.~\ref{fig:AltSourceLogic} shows, the trees reconstructed from PCD
are \emph{logical} trees where the edges represent the amalgamated
versions of the actual physical edges.  For example, the logical tree
labeled $(SLb_3)$ has a direct edge from $b_3$ to $b_6$, whereas the
actual route, shown in $(Sb_3)$ passes through an internal node.
Since this node does not feature as a junction in the tree $(Sb_3)$,
it will not be detected from the PCD.  Moreover each internal vertex
in the original graph has multiple appearances in the logical trees
with no identifying information attached to them.  Correctly
identifying multiple representation of the same internal node will be
the central challenge of this work.

\begin{figure}[t]
  \centering
  \includegraphics[scale=0.8]{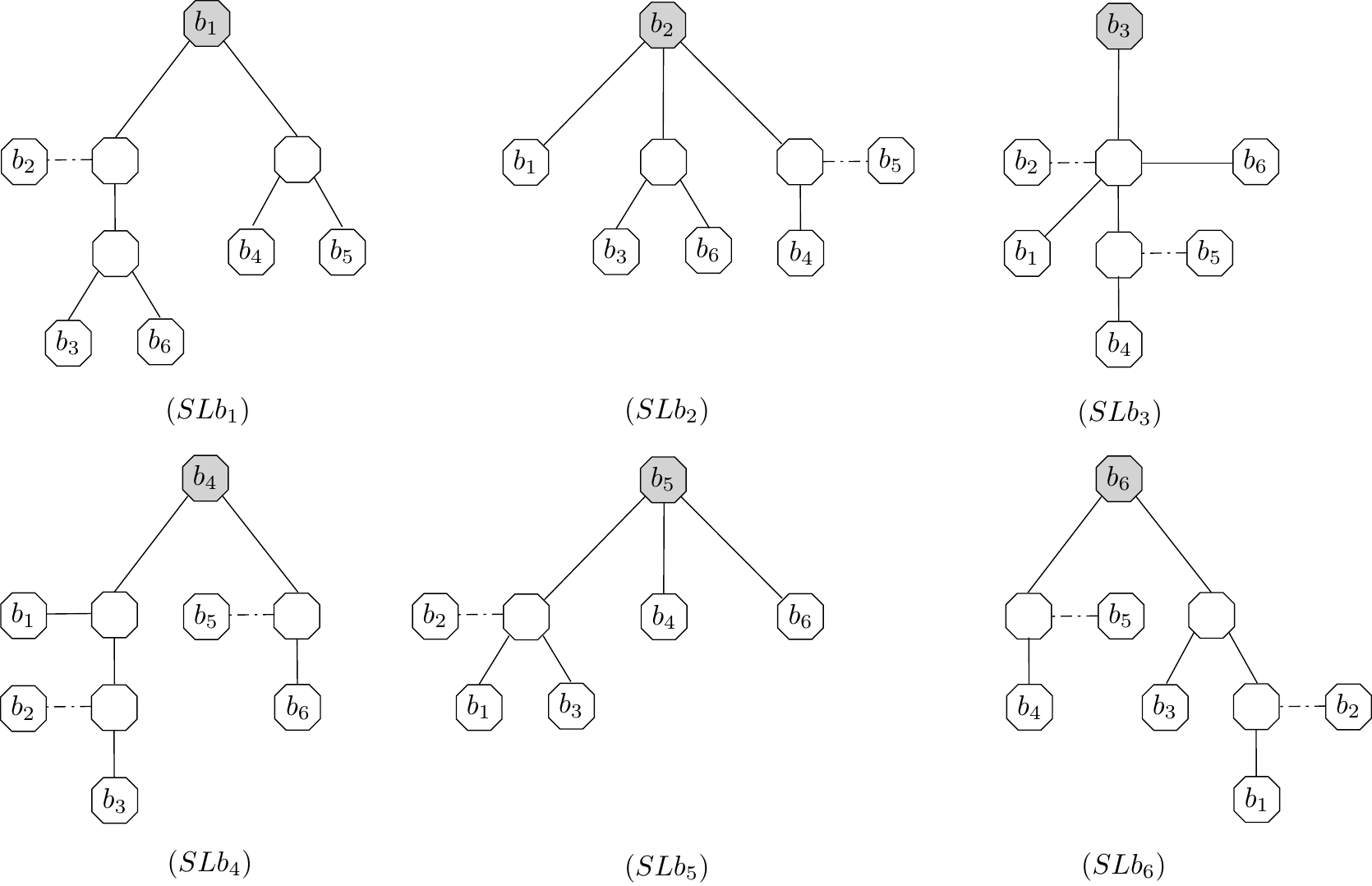}
  \caption{Representation of PCD on the network graph through the set of
    observed logical source trees for sources $b_1$ to $b_6$.}
  \label{fig:AltSourceLogic}
\end{figure}

The information form logical source trees
(Fig.~\ref{fig:AltSourceLogic}) is only enough for reconstruction of
special class of network graph, namely the symmetric one where both
the routing and edge weights are symmetric.  When this is not the
case, the
measurements of the form $\text{PCD}(b_1,b_2 \succ b)$ will be essential to
reconstruct \emph{logical receiver trees} shown in 
Fig.~\ref{fig:AltRoutingLogicalT2}.  Those contain information about
the paths with
the selected receiver $b$ and with the source set $V^B \setminus
\{b\}$.

\begin{figure}[t]
  \centering
  \includegraphics[scale=0.8]{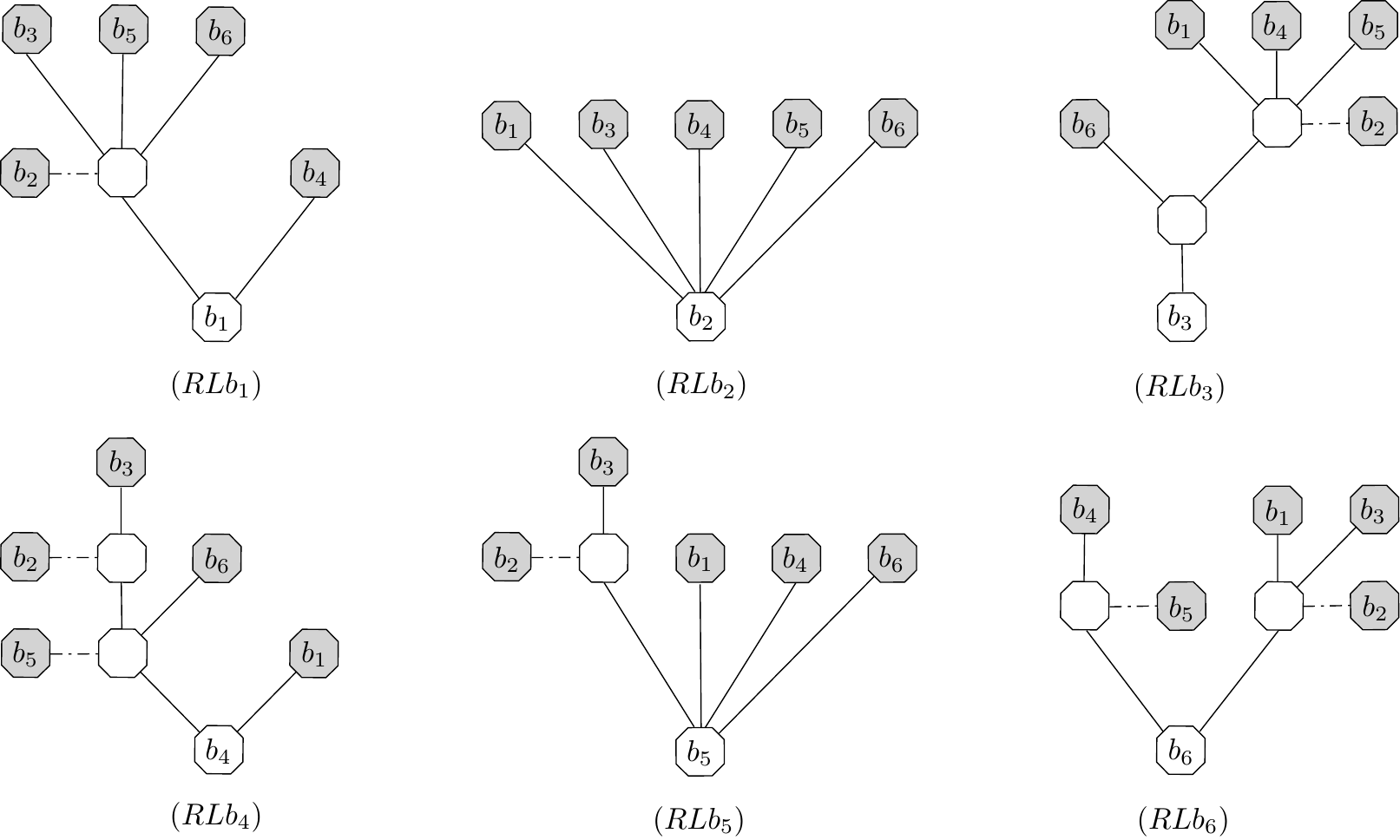}
  \caption{Representation of PCD on the network graph through the set of
    observed logical receiver trees for receivers $b_1$ to $b_6$.} 
  \label{fig:AltRoutingLogicalT2}
\end{figure}

To summarize, the main goal in this paper is to develop a practical
algorithm to reconstruct the original graph from set of measurements
information of the form $\text{PCD}(b_1 \prec b_2, b_3)$ and
$\text{PCD}(b_2, b_3 \succ b_1)$ for $b_1,b_2,b_3 \in V^B$ and to
establish necessary and sufficient conditions for the successful
reconstruction.

\subsection{Outline of paper}

The outline of the paper is as follows. In Section~\ref{sec:statement}
we state the definitions and assumptions concerning the network
topology and path measurements.  We then state our main
Theorem~\ref{thm:main} concerning the reconstruction of the general
non-symmetric weighted graph under these assumptions. The proof of
Theorem~\ref{thm:main} will proceed by establishing that a computation
codified as Algorithm~\ref{algorithm:non_symmetric} reconstructs the
topology under the stated assumptions, as established in
Section~\ref{sec:proof_nonsymmetric} for non-symmetric routing.  We
also discuss the operation of Algorithm~\ref{algorithm:non_symmetric}
on graphs that do not satisfy some assumptions of
Theorem~\ref{thm:main}, and the relation of the output to the true
graph.  The flow of Algorithm~\ref{algorithm:non_symmetric} is
slightly modified if the routing is symmetric.  This case is
considered in Section~\ref{sec:proof_symmetric}.  We also present a
modified Algorithm~\ref{algorithm:symmetric}, suitable \emph{only} for
the case of symmetric paths.  We conclude in
Section~\ref{sec:conclusion_futurework} with a discussion of possible
future work.


\section{Problem Statement and the Main Result}\label{sec:statement}
In this section we set up our model and formulate our results.
Informally, we have a graph with some metric assigned to directed
edges and a subset of vertices that is declared to be the ``boundary''
(denoted by $b_1$, $b_2$ etc) where the observations are made.  We
assume there is a fixed path (``route'') between each ordered pair of
boundary vertices (``source'' to ``receiver'').  We further assume
that the metric is such that we can measure the length of any route
and also the length of the common part of any two routes starting at
the same source or any two routes coming into the same receiver.
These assumptions and the reconstruction problem are made precise
below.

\subsection{Problem Setup}
\label{sec:setup}

\begin{defn}
	A network graph $\mathcal{N} = \NetworkGraph$ is an edge-weighted graph $\Graph =
	(V,E,\weight)$ together with a set $V^B\subset V$ of
	\emph{boundary vertices} and the paths $\Paths$ between
	them.  In detail, 
	\begin{itemize}
		\item $V$ is a finite set of vertices,
		\item $E \subset V\times V$ is the set of directed edges (no loops
		or multiple edges are allowed),
		\item $\weight : E \to \R_+$ are the edge weights,
		\item $V^B$ is an arbitrary subset of $V$
		\item there is a path $\Path(b,b')$ between any pair of boundary
		vertices $b \neq b'$; each path $\Path(b,b')$ is simple and
		assumed to be fixed for the duration of observation of the
		network.  The paths are assumed to have the \emph{tree consistency
			property}: for any $b_1$, $b_1'$, $b_2$ and $b_2'$ in $V^B$ the
		intersection $\Path(b_1,b_1') \cap \Path(b_2,b_2')$ is
		connected.
	\end{itemize}
\end{defn}

\begin{figure}[ht]
	\centering
	\includegraphics[width=0.375\textwidth]{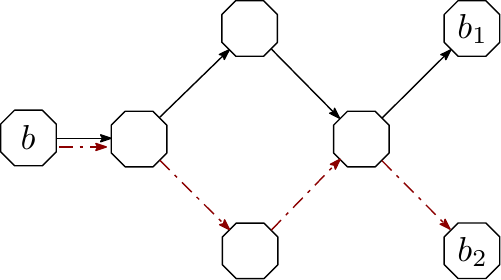}
	\caption{An example of the violation of the tree consistency
		property.  The paths $\Path(b,b_1)$ and $\Path(b,b_2)$ first
		diverge and then meet again.} 
	\label{fig:TreeConst}
\end{figure}

We will assume that a path $\Path(b,b')$ does not pass through any
other boundary vertices.  This is done for convenience only; a graph
can be easily made to satisfy this condition by ``drawing out'' the
boundary vertices from the bulk of the graph as shown in
Fig.~\ref{fig:BvRemove}.  The vertices that do not belong to the
boundary we will call \emph{internal vertices} and use the notation
$V^I = V \setminus V^B$.

\begin{figure}[ht]
	\centering
	\includegraphics[width=0.6\textwidth]{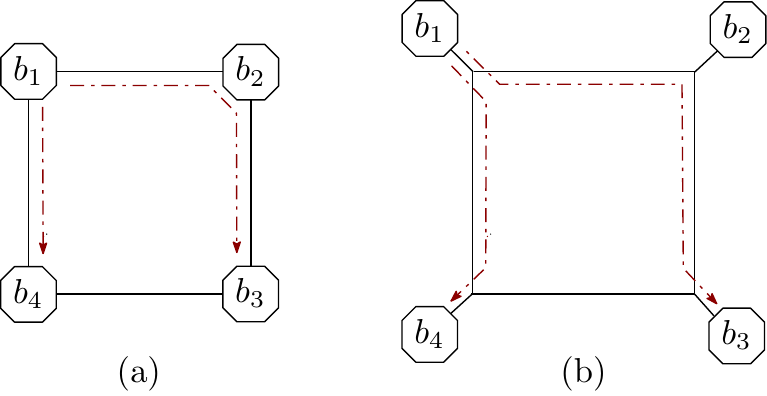}
	\caption{A graph (square with no diagonals) with the path
		$\Path(b_1,b_3)$ going through $b_2$.  The same graph with the
		boundary vertices drawn out.}
	\label{fig:BvRemove}
\end{figure}

We remark that we do not assume that the weights are symmetric:
$\weight_{x,y}$ is generally different from $\weight_{y,x}$.  We also
do not need to assume that the paths are symmetric.  However, since
the latter case is important in applications and allows for a
simplified reconstruction algorithm, we will devote some time to its
separate treatment, in particular in Definition~\ref{def:symm_routing}
and Section~\ref{sec:proof_symmetric}.

Let us consider some implications of the tree consistency property.
Consider two paths, $\Path(b,b_1)$ and $\Path(b,b_2)$ with some
distinct $b,b_1,b_2 \in V^B$.  The tree property implies that the paths
can be written as
\begin{align}
\label{eq:path1}
\Path(b,b_1) &= [b, x_1, x_2, ... , x_j, y_1, y_2, ..., b_1] \\
\label{eq:path2}
\Path(b,b_2) &= [b, x_1, x_2, ... , x_j, z_1, z_2, ..., b_2],
\end{align}    
where the vertex sets $\{y_1, y_2, ...\}$ and $\{z_1, z_2, ... \}$ are
disjoint, see Fig.~\ref{fig:Junction}(a).

Similarly tree consistency property applied to paths 
$\Path(b_1,b)$ and $\Path(b_2,b)$ implies that
\begin{align}
\label{eq:path_b1b}
\Path(b_1,b) &= [b_1, y_1, y_2, ... ,y_{i}, x_1, x_2, ..., b] \\
\label{eq:path_b2b}
\Path(b_2,b) &= [b_2, z_1, z_2, ... ,z_{j}, x_1, x_2, ..., b],
\end{align}    
with disjoint $\{y_1, y_2, ... , y_{m}\}$ and $\{z_1, z_2, ... ,
z_n\}$, see Fig.~\ref{fig:Junction}(b).

\begin{defn}
	The vertex $x_j$ in equations (\ref{eq:path1})--(\ref{eq:path2}) is
	called the $(b \prec b_1,b_2)$-junction. Note that the set
	$\{x_1, ... , x_j\}$ is allowed to be empty in which case $b$ acts
	as the junction.  The vertex $x_1$ in equations
	(\ref{eq:path_b1b})--(\ref{eq:path_b2b}) is called the
	$(b_1,b_2 \succ b)$-junction.
\end{defn}

\begin{figure}[ht]
	\centering
	\includegraphics[width=1\textwidth]{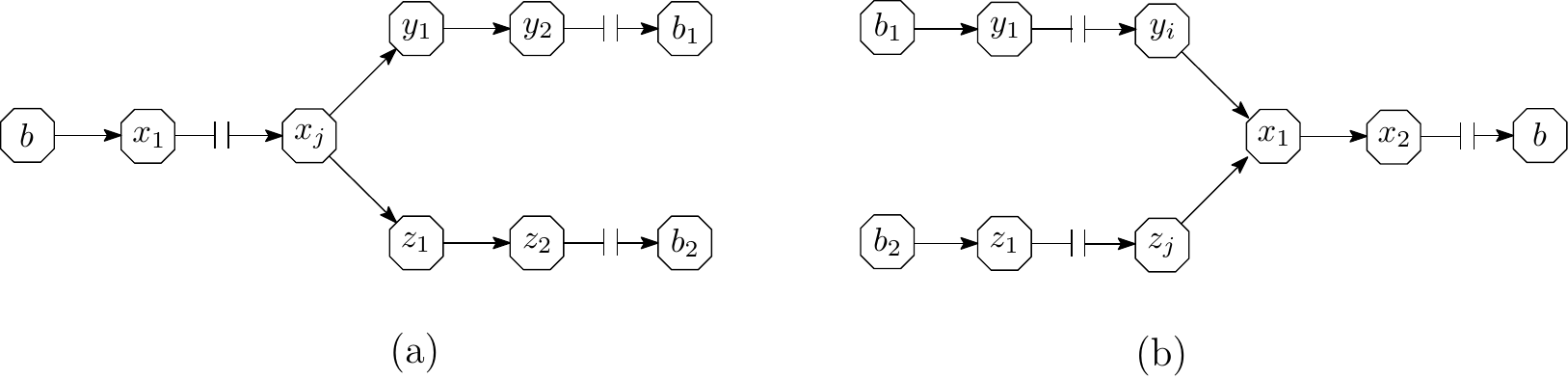}
	\caption{(a) $x_j = (b \prec b_1,b_2)$-junction, 
		(b) $x_1 = (b_1,b_2 \succ b)$-junction.} 
	\label{fig:Junction}
\end{figure}

We remark that we use symbols $\prec$ and $\succ$ not as relational
operators but as separators in the list of 3 vertices which are
pictorially similar to the path configurations in Figure
\ref{fig:Junction}.

To specify the graph reconstruction problem we will be solving we need
to define the set of \emph{measurements} available to us.
The length of a path is defined as the sum of the weights $\weight$ of
its edges; we will denote the length by $|\cdot|$.  We consider a
single vertex as a zero-length path; the length of an empty set is
also set to be zero.  This allows us to assign length to an intersection
of two paths between boundary vertices which is either empty or a single
vertex or a connected subpath.

The totality of the measurements we can make will be called the Path
Correlation Data (PCD).  In includes, for all $b, b_1, b_2 \in V^B$,
\begin{itemize}
	\item the length $|\Path(b,b_1)|$,
	
	\item the length $\left| \Path(b,b_1) \cap \Path(b,b_2)
          \right|$, which we will denote by $\PCD(b \prec b_1,b_2)$,
	
	\item the length $\left| \Path(b_1,b) \cap \Path(b_2,b)
          \right|$, which we will denote by $\PCD(b_1,b_2 \succ b)$.
\end{itemize}

Thus we can directly measure the distance from $b \in V^B$ to any $(b
\prec b_1,b_2)$-junction, or from the $(b_1,b_2 \succ
b)$-junction to $b$.  We can also infer the distance from the $(b
\prec b_1,b_2)$-junction to $b_1$, or from $b_1$ to the $(b_1,b_2\succ
b)$-junction, see Fig.~\ref{fig:JunctionMeas}. 

\begin{figure}[ht]
	\centering
	\includegraphics[width=1\textwidth]{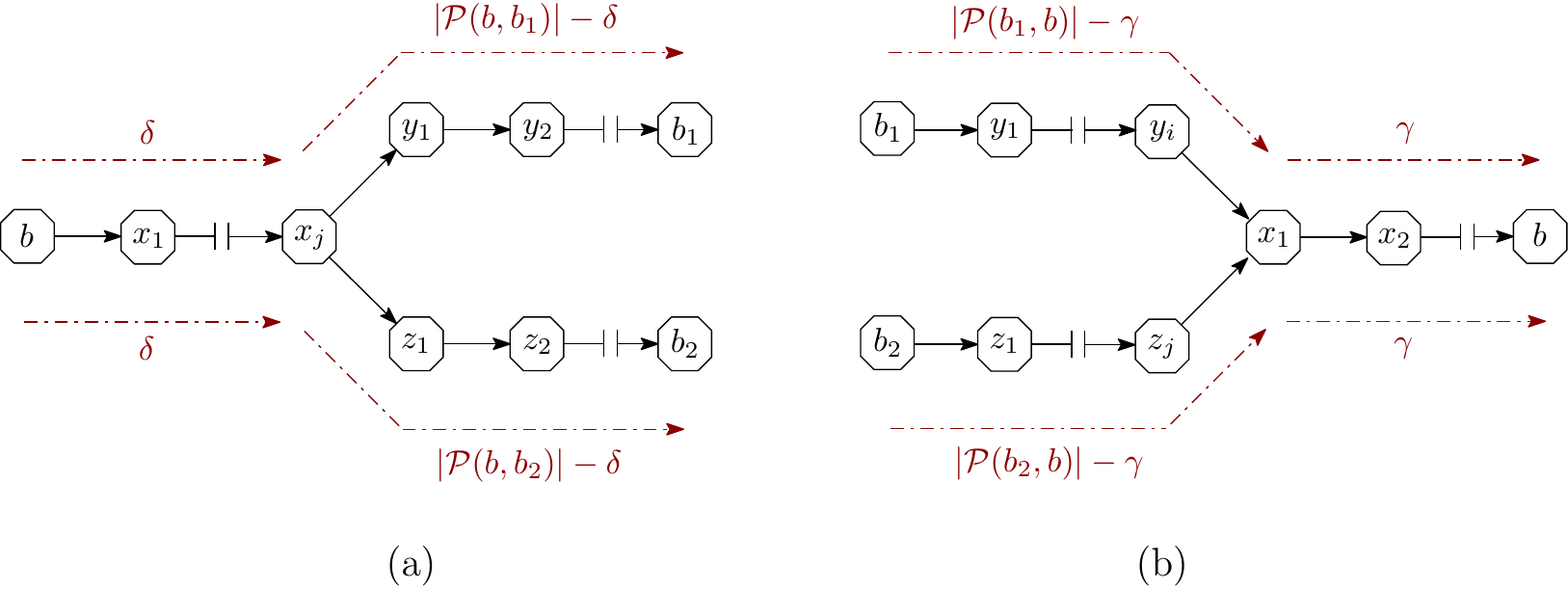}
	\caption{Various distances we can measure from Path
          Correlation Data (PCD).  Here $\delta = \PCD(b, \prec
          b_1,b_2)$ and $\gamma = \PCD(b_1,b_2 \succ b)$.} 
	\label{fig:JunctionMeas}
\end{figure}

Our principal question is thus: \emph{ Which network graphs
  $(\mathcal{G},V^B, \Paths)$ can be reconstructed from their path
  correlation data and how does one accomplish this?}

\subsection{Some obvious necessary conditions}
\label{sec:necessary_examples}

Before we state our result and the associated reconstruction
algorithm, let us consider examples that show some obvious
necessary conditions we need to impose on the network graph
$\mathcal{N}=(\Graph,V^B,\Paths)$ in order to be able to reconstruct it.

\begin{exmp}
	Consider the network graphs in Fig.~\ref{fig:Example_1}, with $V^B =
	\{u,v,w\}$ and with the routing
	paths indicated by dashed lines.  
	None of the routing paths pass through the
	edge $e=(u,w)$ therefore the length of this edge cannot influence
	the Path Correlation Data in any way.  Conversely, the length
        of the edge $e$ (and even its existence) cannot be inferred
        from the PCD.
	\begin{figure}[ht]
		\centering
		\includegraphics[width=0.7\textwidth]{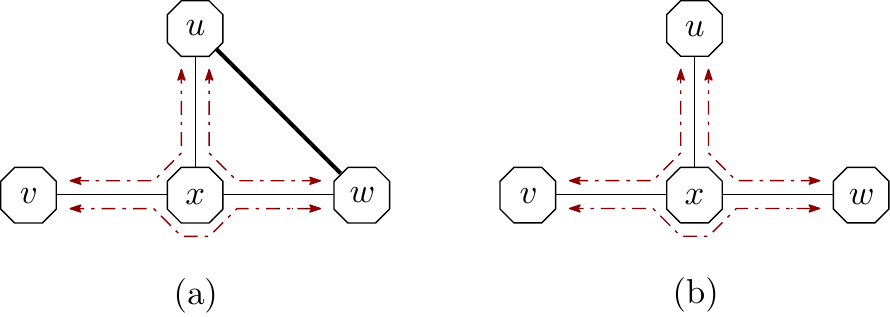}
		\caption{Failure to recover the edge $e=(u,w)$: the graphs (a) and
			(b) will produce the same PCD since none of the paths pass
			through the extra edge.}.
		\label{fig:Example_1}
	\end{figure}
\end{exmp}

\begin{exmp}
	\label{ex:bad_degree2}
	Consider the network graphs in Fig.~\ref{fig:Example_2} with
        boundary vertex set $V^B = \{u,w\}$.  In the left graph the
        length of the edge $(x,u)$ will never appear in the PCD on its
        own, without being added to the length of the edge $(x,w)$.
        Therefore, it will be impossible to reconstruct the location
        of the vertex $x$, and even detect it at all.  This will be
        the case for any internal vertex of degree 2.
	
	\begin{figure}[ht]
		\centering
		\includegraphics[width=0.45\textwidth]{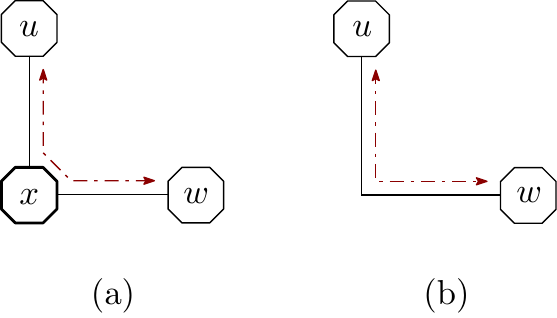}
		\caption{Failure to recover the internal vertex $x$: the graphs
			(a) and (b) will produce the same PCD as long as the sum of the
			lengths of $(u,x)$ and $(x,w)$ in the graph (a) is equal to the
			length of the edge $(u,w)$ in the graph (b).}
		\label{fig:Example_2}
	\end{figure}
\end{exmp}

\begin{exmp}
	Consider the network graphs in Fig.~\ref{fig:Example_3} with the
	boundary vertex set
	\begin{equation}
	V^B = \{u_1,v_1,u_2,v_2\}.      
	\end{equation}
	In Fig.~\ref{fig:Example_3}(a) the paths $\Path(u_1,v_1)$ and
        $\Path(u_2,v_2)$ intersect at an internal vertex $x$, while in
        Fig.~\ref{fig:Example_3}(b) they do not have any vertices in
        common.  However, the two graphs will produce the same PCD and
        will be indistinguishable.
	
	\begin{figure}[ht]
		\centering
		\includegraphics[width=0.9\textwidth]{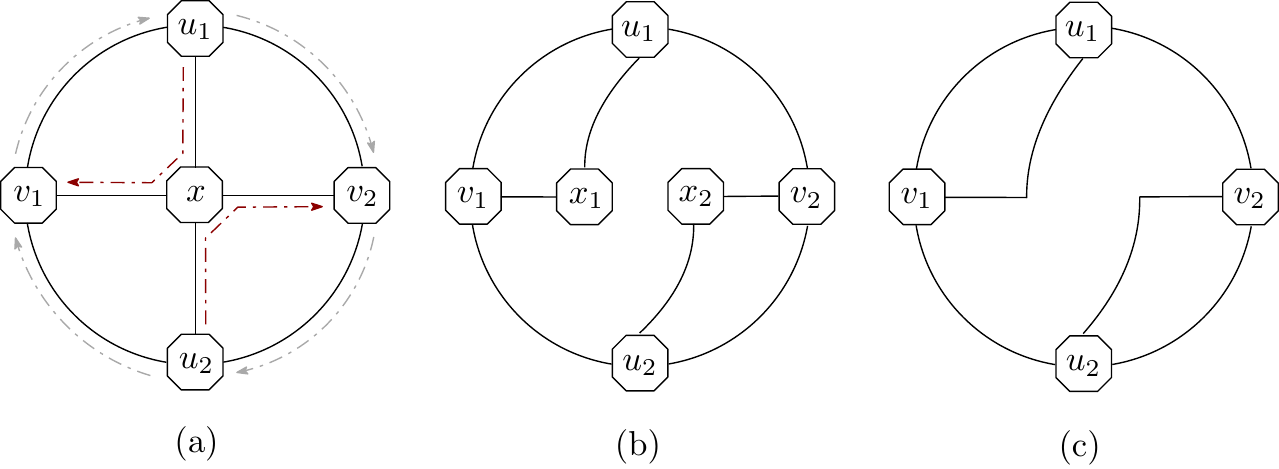}
		\caption{Failure to recover the integrity of the
                  vertex $x$: the graphs (a), (b) and (c) will produce
                  the same PCD since the path between $u_1$ and $v_1$
                  is in no way correlated to the path between $u_2$
                  and $v_2$.}
		\label{fig:Example_3}
	\end{figure}
	
	The reader will undoubtedly observe that the vertices $x_1$ and $x_2$
	in Fig.~\ref{fig:Example_3}(b) will not be detected, and the
        graph in Fig.~\ref{fig:Example_3}(c) is the ``minimal'' graph
        which will have the same PCD.  By making the graph
	structure more complicated one can easily construct an example
	where $x$, $x_1$ and $x_2$ will act as junctions for some pairs of
	paths and thus will be detectable, see
        Fig.~\ref{fig:Example_3_Sep}, but the two graphs are still not
        distinguishable from their PCD.
	
	\begin{figure}[ht]
		\centering
		\includegraphics[width=0.7\textwidth]{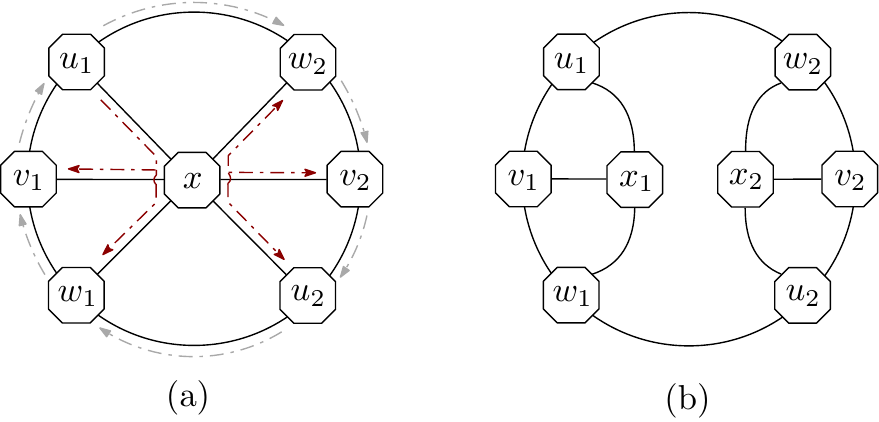}
		\caption{A more complicated example of failure to recover the
			integrity of the vertex $x$.}
		\label{fig:Example_3_Sep}
	\end{figure}
	
	Thus the real problem is that in the left graph in
	Fig.~\ref{fig:Example_3_Sep} there are two families of paths going
	through the internal vertex $x$ that do not interact in any way.
\end{exmp}

\subsection{Statement of the main result}
\label{sec:main_result}

The main result of this paper is that the necessary conditions
illustrated by examples in Section~\ref{sec:necessary_examples} are in
fact sufficient for the reconstruction!  We start by formalizing
(and naming) the conditions we observed.

\begin{defn}
	An edge is called \emph{unused} if there is no path in $\Paths$
	containing it.
\end{defn}

We remark that if there are no unused edges in a network graph, each
internal vertex has at least one incoming and at least one outgoing edge.

\begin{defn}
	An internal vertex $x$ is called \emph{trivial} if it has only one
	incoming and only one outgoing edge (i.e.\ edges of the form
	$(y_1,x)$ and $(x,y_2)$ correspondingly).
\end{defn}

We remark that if there are no unused edges, then there are at least
two paths through every non-trivial internal vertex.

\begin{defn}
	For an internal vertex $x \in V^I$, let $S_x \subset V^B$ to be the set
	of the sources and $R_x \subset V^B$ be the set of the receivers
	whose paths pass through $x$. More precisely, 
	\begin{equation*}
	\begin{split}
	S_x &= \big \{ b \in V^B : \exists \hat{b} \in V^B,\ 
	x \in \Path(b,\hat{b}) \big \} \\
	R_x &= \big \{ \hat{b} \in V^B : \exists b \in V^B,\ 
	x \in \Path(b,\hat{b}) \big \}.
	\end{split}
	\end{equation*}
\end{defn}

\begin{defn}
  \label{def:separable}
  A vertex $x \in V^I$ is called \emph{separable} if there are
  disjoint non-empty partitions $S_x = S_x^1 \cup S_x^2$ and
  $R_x = R_x^1 \cup R_x^2$ with the property that
	\begin{equation}
	\label{eq:separable_def}
	b\in S_x^j,\ \hat{b}\in R_x^{j'} \mbox{ with }j\neq j'
        \quad\Rightarrow\quad x \notin \Path(b,\hat{b}).
	\end{equation}
\end{defn}

An example of a separable vertex is shown in Fig.~\ref{fig:SeparateExam}.  The
partition sets here are $S^1 = \{b_1\}$, $S^2 = \{b_2\}$, $R^1 =
\{b_3,b_4\}$ and $R^2 = \{b_5,b_6\}$.

\begin{figure}[ht]
	\centering
	\includegraphics[width=1\textwidth]{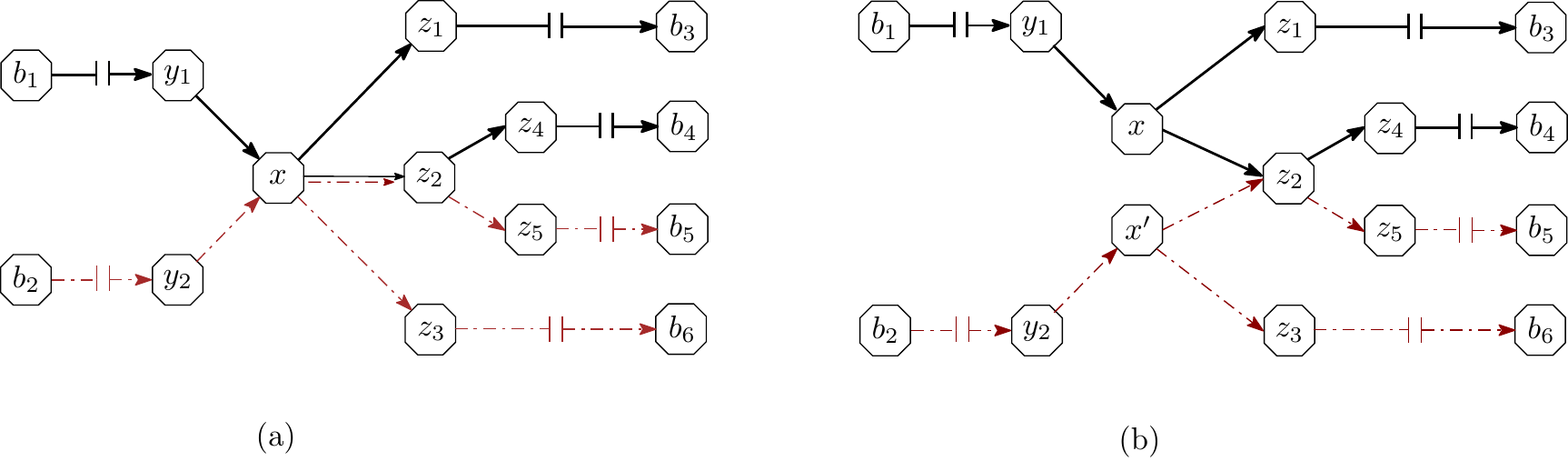}
	\caption{Paths through a separable vertex $x$ and its separation
		into $x_1$ and $x_2$.}
	\label{fig:SeparateExam}
\end{figure}

Finally, if the graph has symmetric routing we need to modify the
conditions slightly but the resulting reconstructibility theorem will
stay the same.  Naturally, symmetric routing is an extra piece of
information and more graphs are reconstructable in this case.  The
natural setting for the problem with symmetric routing is a
non-directed graph, but since we still allow non-symmetric edge
weights, we will keep the edges directed.  As a result, edges come in
pairs which correspond to undirected edges splitting into two directed
ones.  This is formalized in the definition of a ``network graph with
symmetric routing'' below.

\begin{defn}
	\label{def:symm_routing}
	We will say that the network graph $\mathcal{N}=\NetworkGraph$ has
	\emph{symmetric routing} if
	\begin{itemize}
		\item for all $x,y\in V$, $(x,y) \in E$ implies $(y,x) \in E$ and
		\item for all $b,b'\in V^B$, the path $\Path(b, b')$ is the
		reversal of the path $\Path(b',b)$, namely
		\begin{equation}
		\label{eq:reversal}
		\Path(b, b') = [b, x_1, x_2, \ldots, x_j, b']
		\quad \Rightarrow \quad
		\Path(b', b) = [b', x_j, x_{j-1}, \ldots, x_1, b].
		\end{equation}
	\end{itemize}
\end{defn}
	
\begin{defn}
  \label{def:symm_vertex}
	A vertex $x \in V^I$ in a network graph with symmetric routing is
	\emph{trivial} if it has two (or less) adjacent vertices. 
	A vertex $x \in V^I$ in a network graph with symmetric routing is
	\emph{separable} if there is a disjoint partition $S_x = S_x^1 \cup S_x^2$
	so that
	\begin{equation}
	\label{eq:separable_def_symm}
	b_1 \in S_x^1,\ b_2 \in S_x^2 
	\quad \Rightarrow \quad
	x \notin \Path(b_1,b_2).
	\end{equation}
	
\end{defn}

We can now state our Main Theorem.  We stress that the statement of
the theorem applies uniformly to network graphs with or without
symmetric routing, the differences being absorbed by the definitions
above.  We will still need to provide two separate (but similar!)
proofs.

\begin{thm}[Main Theorem]
	\label{thm:main}
	Let $(\Graph,V^B,\Paths)$ be a network graph. If
	\begin{enumerate}
		\item \label{cond:edges}
		no edge $e \in E$ is unused,
		\item \label{cond:degrees}
		no $x \in V^I$ is trivial,
		\item \label{cond:nonsep} 
		no $x \in V^I$ is separable,
	\end{enumerate}
	then $(\Graph,V^B,\Paths)$ is uniquely reconstructable from
	its Path Correlation Data (PCD).
	
	To put it more generally, in every class of network graphs with the
	same PCD, there is a unique network graph which satisfies the above
	conditions.
\end{thm}

The theorem will be proved constructively, by presenting a
reconstruction algorithm and verifying its result.  The second part,
which posits not only uniqueness but also the existence of the
reconstructed graph, means, in practical terms, that even when given
PCD from a graph that does not satisfy the conditions, the algorithm
will terminate and produce a ``nearby'' result which does.

\subsection{Comments on the algorithm for non-symmetric routing}

The algorithm for the case of non-symmetric routing
(Algorithm~\ref{algorithm:non_symmetric}) works by discovering the
internal vertices and reconstructing the routing paths in the format
\begin{equation}
\label{eq:reconstructed_path_format}
\RPath(b, \hat{b}) = [(b,0),\,(x_1,\delta_1),\,(x_2,\delta_2),
\ldots,\, (\hat{b},\delta)],
\end{equation}
where $\delta_i$ is the cumulative distance from $b$ to $x_j$ along the
path (naturally, $\delta = |\Path(b, \hat{b})|$).  Once every path is
complete, the edges can be read off as pairs of consecutive vertices
appearing in a path.  The internal vertices
are discovered as junctions.  The main difficulty lies in identifying
different junctions that correspond to the same vertex.  This is done
by a depth-first search in the function {\sc UpdatePath}.

\algnewcommand{\IfThen}[2]{%
  \State \algorithmicif\ #1\ \algorithmicthen\ #2}
\algnewcommand{\IfThenElse}[3]{
  \State \algorithmicif\ #1\ \algorithmicthen\ #2\ \algorithmicelse\ #3}

\begin{algorithm}[t]
	\caption{Reconstruction of the network graph
		$\NetworkGraph$}
	\label{algorithm:non_symmetric}
	\begin{algorithmic}[1]
		\For{$b_1,b_2 \in V^B$} \Comment{{\bf Initialization}}
		\State $\RPath(b_1,b_2) = [ (b_1,0),(b_2, |\Path(b_1,b_2)|) ]$
		\EndFor
		
		\For{$b_1,b_2,b_3 \in V^B$}  \Comment{{\bf Main Loop}}
		\State create label $a$ \label{algo:create_x}
		\State $\delta = \PCD(b_1 \prec b_2,b_3)$ \label{algo:distance_x}
		\State {\sc UpdatePath}($\RPath(b_1,b_2), a,\delta$) \label{algo:first_call}
                \IfThenElse {``symmetric routing''} {$a'=a$} {create label $a'$} \label{algo_create_y}
		\State $\delta' = |\Path(b_2,b_1)| - \PCD(b_2,b_3 \succ b_1)$ \label{algo:dist_y} 
		\State {\sc UpdatePath}($\RPath(b_2,b_1), a',\delta'$) \label{algo:second_call}
		\EndFor
		
		\State \textbf{Read off} the graph from reconstructed paths
		$\RPath$. \Comment{{\bf Return the result}}
		
		\Statex
		\Function{UpdatePath}{$\RPath(u,v),a,\delta$}
		\Comment{{\bf Recursive Function}} \label{algo:function_start}
		\IfThen {$\exists\,(\cdot,\delta) \in \RPath(u,v)$} {return}
		\label{algo:check_vertex}
		\State insert $(a,\delta)$ into $\RPath(u,v)$
                \State $\delta' = |\Path(u,v)| - \delta$
		\For {$z \in V^B$} \label{algo:discovery_loop}
		\IfThen {$\PCD(u \prec v,z) \geq \delta$} 
		\label{algo:transfer_x_tail}
		{\sc UpdatePath}($\RPath(u,z),a,\delta$)
		\IfThen {$\PCD(u,z \succ v) \geq \delta'$} 
		\label{algo:transfer_x_head} 
                {\sc UpdatePath}($\RPath(z,v),a,|\Path(z,v)| - \delta'$)
		\EndFor
		\EndFunction
	\end{algorithmic}
\end{algorithm}

The following comments might be in order.  In
lines~\ref{algo:create_x}--\ref{algo:distance_x} $a$ is the label for
the vertex that is the $(b_1 \prec b_2,b_3)$-junction and $\delta$ is the distance from $b_1$ to $a$ along the path $\Path(b_1,b_2)$.  In
lines~\ref{algo_create_y}--\ref{algo:dist_y} $a'$ is the label for the
$(b_2,b_3 \succ b_1)$-junction and $\delta'$ is the distance from $b_2$ to $a'$ along the path $\Path(b_2,b_1)$.  If we know {\it a priori\/} that the routing is symmetric, the $(b_1 \prec b_2,b_3)$-junction and the $(b_2,b_3 \succ b_1)$-junction are the same vertex and can receive the same label.  

Line~\ref{algo:check_vertex} checks if there is already a vertex
distance $\delta$ from $b_1$ (if there is, label $a$ is unused).
Finally, the loop starting on line~\ref{algo:discovery_loop} looks for
any other paths that the vertex with label $a$ must belong to. Here we rely heavily on the tree consistency property.

\begin{figure}[h]
	\centering
	\includegraphics[width=0.8\textwidth]{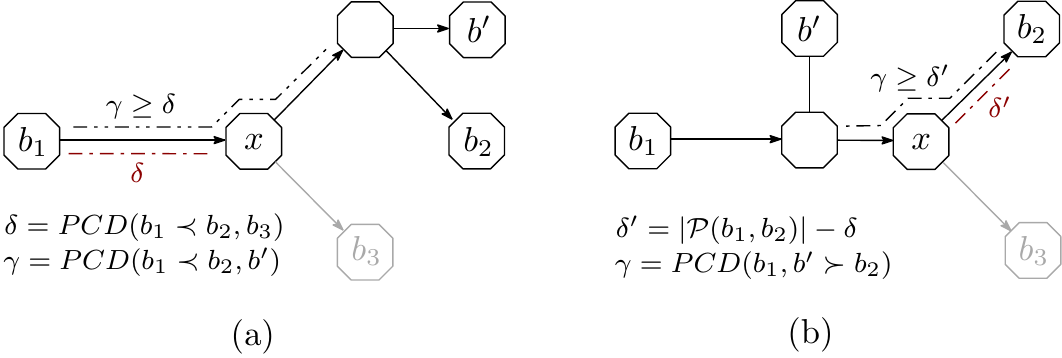}
	\caption{Insertion of internal vertex $x$ in the reconstructed paths, (a) $\RPath(u,z)$, and (b) $\RPath(z,v)$ by fulfillment of conditions~\ref{algo:transfer_x_tail} and~\ref{algo:transfer_x_head} respectively called by line~\ref{algo:first_call} of Algorithm~\ref{algorithm:non_symmetric}.}
	\label{fig:VertexInsertion}
\end{figure}

Some further code improvements are possible.  Creating and then
discarding unused labels can be avoided either by performing a check
similar to line~\ref{algo:check_vertex} in the main loop or, more
elegantly, by making $a$ an optional argument to the function {\sc
UpdatePath} and creating a label after line~\ref{algo:check_vertex}
if no $a$ was supplied.  Additionally, if the edge weights are
symmetric, the call to $\PCD$ in line \ref{algo:dist_y} can be avoided
by using $\PCD(b_2,b_3 \succ b_1) = \PCD(b_1 \prec b_2,b_3) = \delta$.

Finally, a crude upper bound on complexity of the algorithm (in terms
of label insertions into various $\RPath$) is $|V^I| \times |V^B|^2$
i.e. the product of number of internal vertices and the square number
of boundary vertices of the graph.


\section{Proof of the Reconstruction: Non-Symmetric Paths}
\label{sec:proof_nonsymmetric}
The proof of Theorem~\ref{thm:main} has three parts, with very similar
arguments in each part.  To facilitate the proof, we first state and
prove an auxiliary lemma.

\begin{lem}
	\label{lem:nondisjoint}
	Let $x$ be an arbitrary \emph{non-separable} internal vertex
        and let $A: V^B \times V^B \to \{\mathrm{T}, \mathrm{F}\}$ be
        a Boolean property (predicate) that is defined on the pairs
        $\left(b,b'\right)$ such that $x\in\Path(b,b')$.  Assume $A$
        is non-constant (i.e. true on some paths and false on
        some others). Define $S_x^1$ to be the set of the sources of the
        paths through $x$ for which $A$ is true and $S_x^2$ to be the
        set of the sources of the paths for which $A$ is false.
        Define $R_x^1$ and $R_x^2$ analogously.  More formally,
	\begin{equation}
	\label{eq:partitions}
	\begin{split}
	S_x^1 &= \left\{ b\in S_x : \exists b' \in R_x \, 
	\Big[ x \in  \Path(b,b')
	\, \wedge \, 
	A\left(b,b'\right) \Big]
	\right\},\\ 
	S_x^2 &= \left\{ b\in S_x : \exists b' \in R_x \, 
	\Big[ x \in  \Path(b,b')
	\, \wedge \, 
	\neg A\left(b,b'\right) \Big] 
	\right\},\\
	R_x^1 &= \left\{ b'\in R_x : \exists b\in S_x \, 
	\Big[ x \in  \Path(b,b')
	\, \wedge \, 
	A\left(b,b'\right) \Big]
	\right\},\\ 
	R_x^2 &= \left\{ b'\in R_x : \exists b\in S_x \, 
	\Big[ x \in  \Path(b,b')
	\, \wedge \, 
	\neg A\left(b,b'\right) \Big]
	\right\}.
	\end{split}
	\end{equation}
	Then $S_x^1 \cap S_x^2$ and
	$R_x^1 \cap R_x^2$ cannot both be empty.
\end{lem}

\begin{proof}
  Since $A$ is not always false, the sets $S_x^1$ and $R_x^1$ are
  non-empty; since $A$ is not always true, $S_x^2$ and $R_x^2$ are
  non-empty.  Furthermore, it is easy to see that
  \begin{equation}
    \label{eq:unionSandR}
    S_x^1 \cup S_x^2 = S_x \qquad\mbox{and}\qquad
    R_x^1 \cup R_x^2 = R_x.
  \end{equation}
  Assume that $S_x^1 \cap S_x^2 = R_x^1 \cap R_x^2 = \emptyset$.  Then
  we are in a position to use non-separability of the vertex $x$ and
  to conclude that there is a path (without loss of generality) from
  $b_1 \in S_x^1$ to $b_2 \in R_x^2$.  But this path either has
  property $A$ or it does not.  In the former case, $b_2 \in R_x^1$
  and in the latter $b_1 \in S_x^2$, contradicting the assumption of
  disjointedness.
\end{proof}


\begin{proof}[Proof of Theorem~\ref{thm:main}: unique reconstructability]
	
	We will now verify that, given the PCD from a graph that satisfies
	the conditions of the Theorem, the algorithm will produce the
	correct reconstruction.  It is straightforward to check that the
	algorithm places a label for a vertex $x$ only in the
	reconstructions of paths that actually contain $x$ and with the
	right value of the cumulative distance $\delta$.  Thus it remains to
	show that
	\begin{itemize}
		\item every vertex $x$ has at least one label created for it
		\item the reconstructed paths are not missing any vertices.
		\item no more than one label is created for each vertex
	\end{itemize}
	Then we will read off sequential pairs of vertices from the
	reconstructed paths to identify edges.  Since there are no missing
	vertices in $\RPath$, the edges thus reconstructed will
	correspond to actual edges in $E$.  By condition~\ref{cond:edges} of
	the theorem, every edge will appear in at least one $\RPath(u,v)$ and
	will therefore be reconstructed.
	
	Let $x$ be an arbitrary internal vertex.  We would like to show the
	algorithm will create a label for $x$ and place it in some
	reconstructed path containing $x$.
	
	Fix an arbitrary path through $x$ and denote the edges that
        the path visits while going through $x$ by
        $e_1 = (x_1,x)$ and $e_2=(x,x_2)$.  We will now use
        Lemma~\ref{lem:nondisjoint} with the property
        $A= A\left(b,b'\right)$ defined as the statement ``the path
        $\Path(b,b')$ passes through \emph{both} $e_1$ and $e_2$'',
        or, in other words
	\begin{equation}
	\label{eq:firstProperty}
	A\left(b,b'\right) = 
	\text{``$\Path(b,b')$ can be written as 
		$[b,\ldots,x_1,x,x_2,\ldots,b']$''}.
	\end{equation}
	There is at least one path on which $A$ is true.  Since $x$ is
        non-trivial and each of its incident edges belongs to at least
        one simple path, there is at least one other path passing
        through $x$ and \emph{not} containing both $e_1$ and $e_2$.
        Therefore we can apply Lemma~\ref{lem:nondisjoint} and
        conclude that one of the pairs $S_x^1$ and $S_x^2$ or $R_x^1$
        and $R_x^2$ are not disjoint.
	
	Without loss of generality, consider a boundary vertex
        $b \in S_x^1 \cap S_x^2$.  Let $\Path_1(b,b_1)$ be a path
        containing both $e_1$ and $e_2$ and $\Path_2(b,b_2)$ be a path
        that passes through $x$ but does not contain both $e_1$ and
        $e_2$.  Since both $\Path_1$ and $\Path_2$ contain $b$ and
        $x$, they must coincide from $b$ to $x$ by the tree
        consistency property.  Therefore, $\Path_2$ contains $e_1$ and
        can not contain $e_2$.  We conclude that $\Path_1$ and
        $\Path_2$ diverge exactly at $x$, see
        Fig.~\ref{fig:NonSymProof_1}.  In other words, $x$ is the
        $(b \prec b_1, b_2)$-junction and a label will be created for
        it in the main loop.  This label will be placed into
        $\RPath(b',b_1)$ unless there is already another label
        corresponding to $x$ there.

	\begin{figure}[ht]
		\centering
		\includegraphics[width=0.4\textwidth]{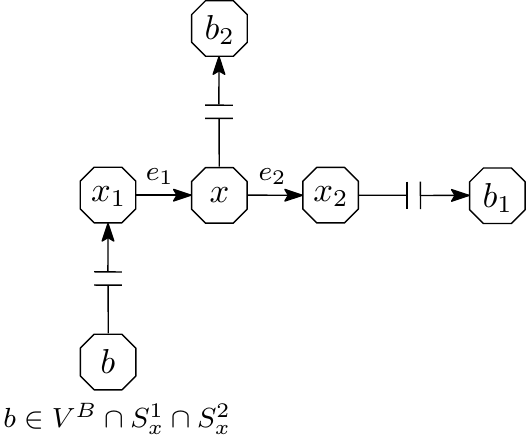}
		\caption{Two paths through a vertex $x$. The path
                  $\Path(b,b_1)$ contains both $e_1$ and $e_2$ and the
                  $\Path(b,b_2)$ does not.}
		\label{fig:NonSymProof_1}
	\end{figure}

	We will now prove that each reconstructed path contains labels for all
	vertices composing the actual path.  Assume the contrary:
        there is a path whose reconstruction does not contain a label
        for an internal vertex $x$. Define
	the property $A$ by
	\begin{equation}
	\label{eq:secondProperty}
	A\left(b,b'\right) = 
	\text{``$\exists a(x) \in \RPath(b, b')$''}.
	\end{equation}
	We have already proved that a label for $x$ will be placed in
        \emph{some} path; together with our assumption it means that
        $A$ is non-constant and we can apply
        Lemma~\ref{lem:nondisjoint}.
	
	If there is a vertex $b \in S_x^1 \cap S_x^2$ then we can find
	$b_1', b_2' \in V^B$ such that $\Path(b,b_1')$ and $\Path(b,b_2')$
	both pass through $x$ but $\RPath(b,b_1')$ has a label corresponding
	to $x$ and $\RPath(b,b_2')$ does not.  By the tree consistency
	property, the two paths coincide from $b$ to at least $x$, therefore
	$\text{PCD}(b \prec b_1', b_2') \geq \delta$, where $\delta$ is the
	distance from $b$ to $x$ along the path $\Path(b,b_1')$.  However, the
	label $a(x)$ was placed into $\RPath(b,b_1')$ by a call to {\sc
	UpdatePath} with this $\delta$.  This would trigger 			the
	condition on line~\ref{algo:transfer_x_tail} with $w = b_2'$ and the
	same label would be placed into $\RPath(b,b_2')$, a contradiction.
	The case when $R_x^1 \cap R_x^2$ is non-empty is treated similarly but
	this time condition on line~\ref{algo:transfer_x_head} ensures the
	transfer of the label.
	
	
	The last step of the proof of reconstruction is to check that only one
	label is created for any vertex.  Assume the contrary, there is a
	vertex $x$ which has at least two different labels corresponding to it
	appearing in different reconstructed paths.  Fix one of the labels
	$a(x)$ and define property $A$ by
	\begin{equation}
	\label{eq:thirdProperty}
	A\left(b,b'\right) = 
	\text{``$a(x) \in \RPath(b, b')$''}.
	\end{equation}
	This definition of $A$ is very similar to
	eq.~\eqref{eq:secondProperty}, but here $a(x)$ is some fixed
	label, whereas in \eqref{eq:secondProperty} it was \emph{any} label of
	$x$.
	
	By Lemma~\ref{lem:nondisjoint} and without loss of generality,
        there is a vertex $b \in S_x^1 \cap S_x^2$.  This means that
        there is a path $\Path(b,b_1)$ with the label $a(x)$ and a
        path $\Path(b,b_2)$ with a label $a'(x)$ different from $a(x)$
        (we remark that we have already shown that each reconstructed
        path contains labels for all its vertices).  As before, the
        two paths coincide from $b$ to at least $x$, therefore
        $\text{PCD}(b \prec b_1', b_2') \geq \delta$, where $\delta$
        is the distance from $b$ to $x$ along the path
        $\Path(b,b_1')$.  During the execution of the algorithm, one
        of the labels was placed in its respective path first.  But
        then the condition on line~\ref{algo:transfer_x_tail} would be
        triggered and the same label would be copied to the other
        path, before the other label is created.  We have thus arrived
        to a contradiction.
\end{proof}

\begin{proof}[Proof of Theorem~\ref{thm:main}: reconstruction of a
	non-compliant graph]
	
	We will now consider the output of the algorithm in case the PCD was
	created by the network graph $\NetworkGraph$ that violates some of
	the conditions of the Theorem.  We will show that there is a
	``compliant'' network graph that has the same PCD and which will
	therefore serve as the output of the algorithm.
	
	We start from the network graph $\mathcal{N} = \NetworkGraph$
        and apply the following ``cleaning'' operations to them (the
        order is important): (1) remove all unused edges, (2) split
        each separable vertex into 2 or more non-separable
        vertices\footnote{See Remark~\ref{rem:max_splitting} below for
          a detailed discussion of why such a splitting exists.},
        (3) remove each trivial vertex by merging its incident edges
        into one edge.  The only adjustments needed to the paths
        $\Paths$ is to choose the correct copies of the split
        vertices.  The set of boundary vertices remains the same.
	
	An example of performing these operations is shown in
	Fig.~\ref{fig:CleaningExam}. In particular, the internal vertex $x$ 		fails
	the non-separability condition, which is seen by defining the
	disjoint partitions
	\begin{align*}
	S = S^1 \cup S^2 = \{ b_1,b_3 \} \cup \{b_2\} \\
	R = R^1 \cup R^2 = \{ b_2 \} \cup \{b_1, b_3\}
	\end{align*}    
	
	\begin{figure}[ht]
	\centering
	\includegraphics[width=0.55\textwidth]{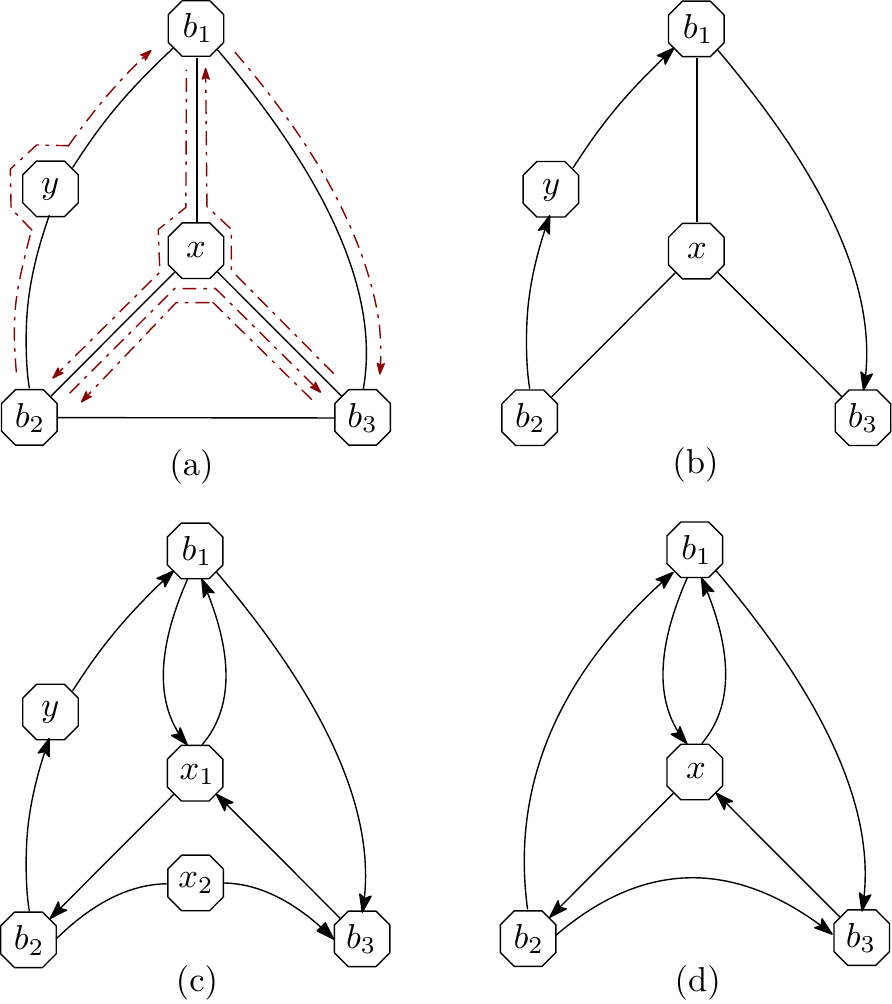}
	\caption{Cleaning of a graph: (a) the original graph (b) removing
		unused edge $e=(b_2,b_3)$ (c) splitting the vertex $x$ into
		$x_1$ and $x_2$, (d) removing trivial vertices $y$ and $x_2$.
		The edges that have no direction marked on them run in both
		directions.}
	\label{fig:CleaningExam}
	\end{figure}
	
	We remark that splitting a separable vertex may require duplicating
	some incoming or outgoing edges, see Fig.~\ref{fig:SeparateExam}, in
	which case the weights get duplicated too.  The edges will be
	duplicated only if both resulting edges are present in some paths;
	thus no unused edges will be created.
	
	Suppose an internal vertex got split into $x_1$ and $x_2$.
	It follows from the definition of the separable vertex that if
	$x\in\Path(b,b_1)$ and $x\in\Path(b,b_2)$ before the split, then
	after the split either both paths contain $x_1$ or both of them
	contain $x_2$.  Therefore the length of each intersection of the
	form $\Path(b,b_1)\cap \Path(b,b_2)$ remains unchanged after a
	split.  The same applies to any pair of paths $x\in\Path(b_1,b)$ and
	$x\in\Path(b_2,b)$.  We conclude that the PCD remains unchanged by
	the operations above.
	
	We denote the network graph so obtained by $\mathcal{N}^C$ (for
	``compliant'' or ``cleaned'').  It is easy to see that
	$\mathcal{N}^C$ satisfies the conditions of the Theorem and will
	therefore be reconstructed from its PCD by the algorithm.  However,
	the PCD of $\mathcal{N}^C$ is the same as the PCD of $\mathcal{N}$.
	Therefore, given the PCD of $\mathcal{N}$, the algorithm will return
	the network graph $\mathcal{N}^C$.  This logic is illustrated by
	Fig.~\ref{fig:CleanChart}.
	
	\begin{figure}[ht]
		\centering
		\includegraphics[width=0.4\textwidth]{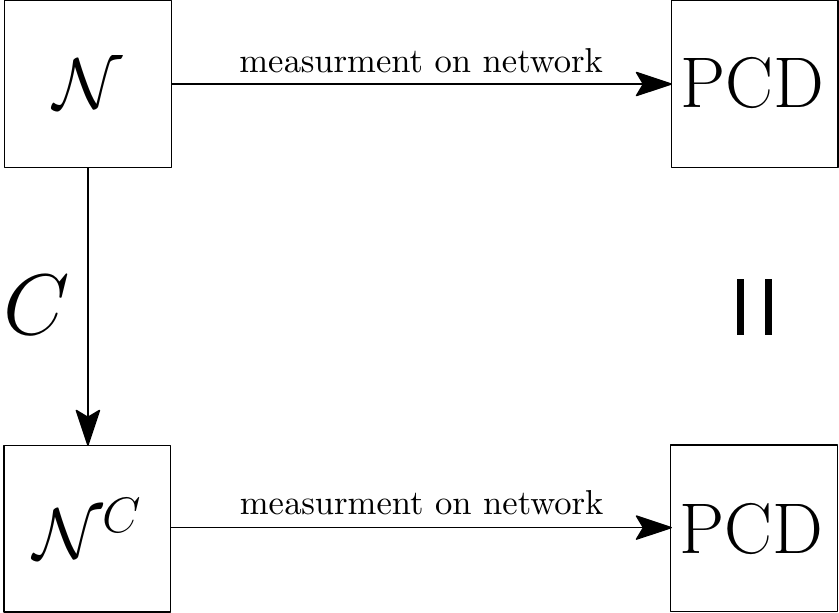}
		\caption{Flow of the proof of the reconstruction of a
			non-compliant graph: $C$ is the operations that produce the
			corresponding compliant network graph $\mathcal{N}^C$, the Path
			Correlation Data of both graphs is the same.}.
		\label{fig:CleanChart}
	\end{figure}
\end{proof}

\begin{remark}
  \label{rem:max_splitting}
  In the proof above we implicitly assumed that there exists a unique
  maximal splitting of a vertex into non-separable parts.  We outline
  the proof of this fact.  Fix $x$ and define equivalence relations
  $\overset{S}{\sim}$ and $\overset{R}{\sim}$ on $S_x$ and $R_x$
  correspondingly by letting
  \begin{align}
    \label{eq:S_equiv}
    b_1 \overset{S}{\sim} b_2
    \qquad&\mbox{if}\qquad
    \exists \hat{b}:\ x\in\Path(b_1,\hat{b}) 
    \mbox{ and }x\in\Path(b_2,\hat{b}),\\
    \hat{b}_1 \overset{R}{\sim} \hat{b}_2
    \qquad&\mbox{if}\qquad
    \exists b:\ x\in\Path(b,\hat{b}_1) 
    \mbox{ and }x\in\Path(b,\hat{b}_2),
  \end{align}
  and closing each one by transitivity.  It is easy to see that they
  are ``dual'' in the following sense: if $x\in\Path(b_1, \hat{b}_1)$ and
  $x\in\Path(b_2,\hat{b}_2)$ then 
  \begin{equation}
    \label{eq:equivalence_duality}
    b_1 \overset{S}{\sim} b_2 
    \quad\Leftrightarrow\quad
    \hat{b}_1 \overset{R}{\sim} \hat{b}_2.
  \end{equation}
  As a consequence, there is a natural one-to-one correspondence
  between the equivalence classes of $\overset{S}{\sim}$ and the
  equivalence classes of $\overset{R}{\sim}$.  These equivalence
  classes provide the finest separation of the vertex $x$.  To be more
  precise, the disjoint partitions
  \begin{align}
    S_x = [b_1]_S \cup [b_2]_S \cup \ldots \cup [b_m]_S,\\
    R_x = [\hat{b}_1]_R \cup [\hat{b}_2]_R \cup \ldots \cup
    [\hat{b}_m]_R,
  \end{align}
  satisfy condition~\eqref{eq:separable_def_symm}.  Here $[\cdot]_S$
  and $[\cdot]_R$ denote equivalence classes with respect to $\overset{S}{\sim}$ and
  $\overset{R}{\sim}$ and $x \in \Path(b_j,\hat{b}_j)$ for every $j$.

  Furthermore, if the nonempty partitions
  $S_x = S_x^1 \cup S_x^2$ and $R_x = R_x^1 \cup R_x^2$ satisfy
  Definition~\ref{def:separable} then for any $b\in S_x$, either
  $[b]_S \subset S_x^1$ or $[b]_S \subset S_x^2$, and similarly for $R$.
\end{remark}


\section{Symmetric Paths}
\label{sec:proof_symmetric}
Reconstruction of network graph with the prior knowledge that the routing is symmetric (the edge weights may not be symmetric) has the advantage of being able to reconstruct a wider class of network graphs. In the following we present the proof of unique reconstruction of network graph with symmetric routing if the conditions of the main Theorem are satisfied. The only differences on the required conditions for exact recovery are contained in definition~\ref{def:symm_vertex} for the non-triviality and non-separability of internal vertex. We start with an appropriate modification of Lemma~\ref{lem:nondisjoint}. 

\begin{lem}
	\label{lem:nondisjointsym}
	Let $x$ be an arbitrary \emph{non-separable} internal vertex and let
	$B$ be a non-constant \emph{symmetric} Boolean property (predicate) that is defined on the pairs $\left(b,b'\right)$ such that $x\in\Path(b,b')$ and $x \in \Path(b',b)$.  Define $S_x^1$ to be the set of the sources of the paths through $x$ for which $B$ is true and $S_x^2$ to be the set of the sources of the paths for which $B$ is false.  More formally,
	\begin{equation}
	\label{eq:partitionssym}
	\begin{split}
	S_x^1 &= \left\{ b\in S_x : \exists b_1 \in V^B \, 
	\Big[ x \in  \Path(b,b_1)
	\, \wedge \, 
	B\left(b,b_1\right) \Big]
	\right\},\\ 
	S_x^2 &= \left\{ b\in S_x : \exists b_2 \in V^B \, 
	\Big[ x \in  \Path(b,b_2)
	\, \wedge \, 
	\neg B\left(b,b_2\right) \Big] 
	\right\}
	\end{split}
	\end{equation}
	Then $S_x^1 \cap S_x^2$ cannot be empty.
\end{lem}

\begin{proof}
For symmetric routing the receiver $R_x$ and the source set $S_x =
S_x^1 \cup S_x^2$ coincide.  Assume that $S_x^1 \cap S_x^2 = \emptyset$. Then using the non-separability of vertex $x$, there is a path from $b_1 \in S_x^1$ to $b_2 \in S_x^2$. But this path either has property $B$ or it does not.  
\begin{itemize}
\item In the former case, we use symmetry of $B(b_1,b_2)$ to conclude
  $b_2 \in S_x^1$.
	\item In the latter case $B(b_1,b_2)$ is false, thereby $b_1 \in S_x^2$,
\end{itemize} 
either way, contradicting the assumption of disjointedness.
\end{proof}

Now we are in the position to apply Lemma~\ref{lem:nondisjointsym} to prove the unique reconstruction of network graph with symmetric routing. While the proof steps are very similar to the ones discussed in proving Theorem~\ref{thm:main}, the main difference here is to ensure that the Boolean function $B$ is symmetric.    

\begin{proof}[Proof of unique reconstructability for symmetric routing]
\textbf{Every vertex $x$ has at least one label created for it}. For an internal vertex $x$, fix an arbitrary path $\Path_0$ through $x$ and denote the vertices that the path visits before and after getting to $x$ by $x_1$ and $x_2$. Denote by 
\begin{align*}
    \Lambda_x := \big \{(x_1,x), (x,x_1), (x,x_2), (x_2,x) \big \} \subseteq E
\end{align*}
the subset of edges contains $x$ and either $x_1$ or $x_2$ as the other end vertex. Define the property $B = B(b,b')$ as the statement 
\begin{equation}
\label{eq:firstPropertysym}
B\left(b,b'\right) = 
\text{``$\Path(b,b')$ satisfies  $| \Path(b,b') \cap \Lambda_x | = 2 $''},
\end{equation}
in other words, $B$ is true if the path can be written as a sequence 
\begin{align*}
\Path(b,b') = [b,\ldots,x_1,x,x_2,\ldots,b'] \quad \text{or} \quad \Path(b,b') = [b,\ldots,x_2,x,x_1,\ldots,b'].
\end{align*}

Due to the symmetric routing assumption on the network graph, function $B$ will be true for the reverse path $\Path(b',b)$ as well. Because of our choice of $x_1$ and $x_2$, there are at least two paths on which $B$ is true, namely $\Path_0$ and its reversal. Since vertex $x$ is non-trivial, there exists vertex $x_3$ adjacent to $x$ and a path $\Path'$ passing through $(x,x_3)$ or $(x_3,x)$ on which $B$ is false. Therefore we can apply \ref{lem:nondisjointsym} and conclude that $S_x^1$ and $S_x^2$ are not disjoint. 

Consider a boundary vertex $b \in S_x^1 \cap S_x^2$, and let $\Path_1(b,b_1)$ and $\Path_2(b,b_2)$ be the two paths where $B$ is true and false respectively. By the tree property, the two paths follow same set of edges before diverging exactly at $x$ because $B$ is false on $\Path_2$. As a result $x$ is the $(b \prec b_1,b_2)$-junction and a label will be created for it in the main loop of Algorithm~\ref{algorithm:non_symmetric}. This label will be placed in both reconstructed paths $\RPath(b,b_1)$ and $\RPath(b_1,b)$ unless there is already another label corresponding to $x$ there. 

\textbf{The reconstruction paths are not missing any vertices}. Fix an
internal vertex $x$ and define the property 
\begin{equation}
  \label{eq:secondPropertysym}
  B\left(b,b'\right) = 
  \text{``$\exists a(x)$ such that 
    $a(x) \in \RPath(b, b')$ and $a(x) \in \RPath(b', b)$''},
\end{equation}
which is clearly symmetric by its definition. We emphasis that
$\RPath(b, b')$ and $\RPath(b', b)$ are required to contain the same
label for $x$. We already proved that a label will be created for $x$,
and when the first label is created, it will be placed into a path and
its reversal.  Therefore $B$ is true on some paths.  We will prove
that this $B$ holds for all paths $(b,b')$ containing $x$.  This will
establish not only that the reconstructed paths are not missing any
vertices, but also that the labels in $\RPath(b', b)$ are the same as
in $\RPath(b, b')$.

Assume the contrary: $B$ fails for some path.  We apply Lemma
\ref{lem:nondisjointsym} and let $b \in S_x^1 \cap S_x^2$. Then there
exist $b_1,b_2 \in R_x$ so that a representation $\hat{a}(x)$ exists
in both $\RPath(b,b_1)$ and $\RPath(b_1,b)$ while $B$ is false on
$(b,b_2)$.  The latter implies that $\hat{a}(x)$ cannot be present in
both reconstructed paths $\RPath(b,b_2)$ or $\RPath(b_2,b)$. This can
be summarized as follows:
\begin{multline}
  \label{eq:nonexistance}
  \exists \hat{a}(x) \mbox{ such that } \\
  \hat{a}(x) \in \RPath(b,b_1)
  \mbox{ and }
  \hat{a}(x) \in \RPath(b_1,b)
  \mbox{ and }
  \Big(
    \hat{a}(x) \not\in \RPath(b,b_2) \mbox{ or }
    \hat{a}(x) \not\in \RPath(b_2,b)
  \Big).
\end{multline}
Now without loss of generality, assume that
$\hat{a}(x) \not\in \RPath(b_2,b)$. The same arguments as in the proof of
Theorem~\ref{thm:main} guarantee that
$\text{PCD}(b_1, b_2 \succ b) \geq \delta'$, where $\delta'$ is the
distance from $x$ to $b$.  Therefore the insertion of the
representation $\hat{a}(x)$ into $\RPath(b_2,b)$ will be attempted by
triggering the condition in line~\ref{algo:transfer_x_head} for
$z = b_2$ within the call of the function {\sc
  UpdatePath}$(\RPath(b_1,b),\hat{a},\cdot)$.  Since $\hat{a}(x) \not\in
\RPath(b_2,b)$, we conclude that there is another label $a'(x) \in
\RPath(b_2,b)$, placed there earlier by a call to {\sc
  UpdatePath}$(\RPath(b_2,b),a',\cdot)$.  But then the condition in
line~\ref{algo:transfer_x_head} will be triggered with 
$z = b_1$ and the label $a'$ will be placed into $\RPath(b_2,b)$, in
contradiction to our assumptions.

The last step of the proof is to show that \textbf{no more than one label is created for each vertex}. For a fixed label $a(x)$ of vertex $x$, define the property
\begin{equation}
\label{eq:thirdPropertysym}
B\left(b,b'\right) = 
\text{``$a(x) \in \RPath(b, b')$''}.
\end{equation}
which is symmetric since we already proved that property
\eqref{eq:secondPropertysym} holds for all paths containing $x$. By
Lemma~\ref{lem:nondisjointsym} there exists $b \in S_x^1 \cap S_x^2$
where the representation $a(x) \in \RPath(b,b_1)$ and
$a'(x) \in \RPath(b,b_2)$ are distinct.  But the two paths pass same
set of edges at least up to vertex $x$, i.e.\
$\text{PCD}(b \prec b_1,b_2) \geq \delta$, where $\delta$ is the
distance from $b$ to $x$ along the path $\Path(b,b_1)$. The execution
of algorithm places one of the labels in its respective path first. As
a result, before any new label for $x$ is created, condition on
line~\ref{algo:transfer_x_tail} will have been triggered and copied
the label to other paths. This is contradiction to the assumption of
existence of two different labels of $x$ in the set of reconstructed
paths.
\end{proof}	

Finally, for graphs with symmetric routing where the exact reconstruction conditions are not met, similar discussion as in Theorem~\ref{thm:main} proves the uniqueness of the reconstruction result. The main point here is that cleaning operations preserve the symmetric routing on the graph. Moreover the compliant graph $\mathcal{N}^C$ satisfies the exact reconstruction conditions with same PCD as the one for original graph $\mathcal{N}$.

\subsection{Specialized algorithm for symmetric routing}	

We have shown that the reconstruction
Algorithm~\ref{algorithm:non_symmetric} is universal: it covers both
the general non-symmetric network and also graphs with symmetric
routing. However, having the prior information that the routing on the
network is symmetric makes it possible to call the recursive function
{\sc UpdatePath} less. This is due to the symmetric routing property
that if an internal vertex $x$ is inserted in reconstruction path
$\Path(u,v)$ for $u,v \in V^B$ then this vertex should also inserted
in the reverse path $\Path(v,u)$ with appropriate distance from root
$v$ (this is not generally true for non-symmetric routing case). To
take advantage of this feature we propose
Algorithm~\ref{algorithm:symmetric} as a specialized version of the
reconstruction algorithm for the graphs with symmetric routing.

Compared to the Algorithm~\ref{algorithm:non_symmetric}, the new
algorithm calls the recursive function twice less (as can be seen by
comparing the main loops of the two algorithms). The change in the number
of calls of reconstruction function is compensated by adding line 14
in the new algorithm where the label of internal vertex which is
inserted in the given path will be inserted in the reverse one as well.

It should be pointed out that although less calls of main
reconstruction function makes the algorithm computationally more
effective, from the point of view of mathematical complexity (in terms
of insertions of a vertex into a reconstructed path), the two
algorithms can be considered the same.

\begin{algorithm}[h]
	\caption{Reconstruction of a network graph with symmetric routing}
	\label{algorithm:symmetric}
	\begin{algorithmic}[1]
		\For{$b_1,b_2 \in V^B$} \Comment{{\bf Initialization}}
		\State $\RPath(b_1,b_2) = [ (b_1,0),(b_2, |\Path(b_1,b_2)|) ]$
		\EndFor
		
		\For{$b_1,b_2,b_3 \in V^B$}  \Comment{{\bf Main Loop}}
		\State create label $a$ \label{algo:create_x_sym}
		\State $\delta = \text{PCD}(b_1 \prec b_2,b_3)$ \label{algo:distance_x_sym_f}
		\State $\delta' = \text{PCD}(b_2,b_3 \succ b_1)$ \label{algo:distance_x_sym_r}
		\State {\sc UpdatePath}($\RPath(b_1,b_2),a,\delta, \delta'$)
		\EndFor
		
\State \textbf{Read off} the graph from reconstructed paths
$\RPath$. \Comment{{\bf Return the result}}

\Statex
\Function{UpdatePath}{$\RPath(u,v),a,\delta, \delta'$}
\Comment{{\bf Recursive Function}} \label{algo:function_start_sym}
\IfThen {$\exists\,(\cdot,\delta) \in \RPath(u,v)$} {return}
\label{algo:check_vertex_sym}
\State insert $(a,\delta)$ into $\RPath(u,v)$
\State insert $(a,|\Path(v,u)| - \delta')$ into $\RPath(v,u)$

\For {$z \in V^B$} \label{algo:discovery_loop_sym}
\IfThen {$\PCD(u \prec v,z) \geq \delta$} {{\sc UpdatePath}($\RPath(u,z),a,\delta,\delta'$)}
\label{algo:transfer_x_tail_sym}
\IfThen {$\PCD(v \prec u,z) \geq \gamma'$} {{\sc UpdatePath}($\RPath(v,z),a,|\Path(v,u)| - \delta',|\Path(u,v)| - \delta $)}
\label{algo:transfer_x_head_sym} 
\EndFor

\EndFunction
	\end{algorithmic}
\end{algorithm}

\subsection{Reconstruction example}

In the following example, we will show how having the prior
information of symmetric routing will help to uniquely reconstruct the
network graph shown in Fig.~\ref{fig:Delta}(a). The selected routing
among the boundary vertices $V^B = \{b_1,b_2,b_3\}$ follows the
sequence $\Path(b_i,b_j) = [b_i,x_i,x_j,b_j]$ for $i,j = 1,2,3$. From
the measurement point of view, the PCD on the graph is represented as
the set of observed logical trees in Fig.~\ref{fig:DeltaReconst}.

If there is no information on the symmetry of routing, then we can
not conclude that $a_1 = a_4$, $a_2=a_5$ and $a_3=a_6$.  The resulting
reconstruction will be the graph appearing in Fig.~\ref{fig:Delta}(b).

\begin{figure}[ht]
	\centering
	\includegraphics[scale=0.9]{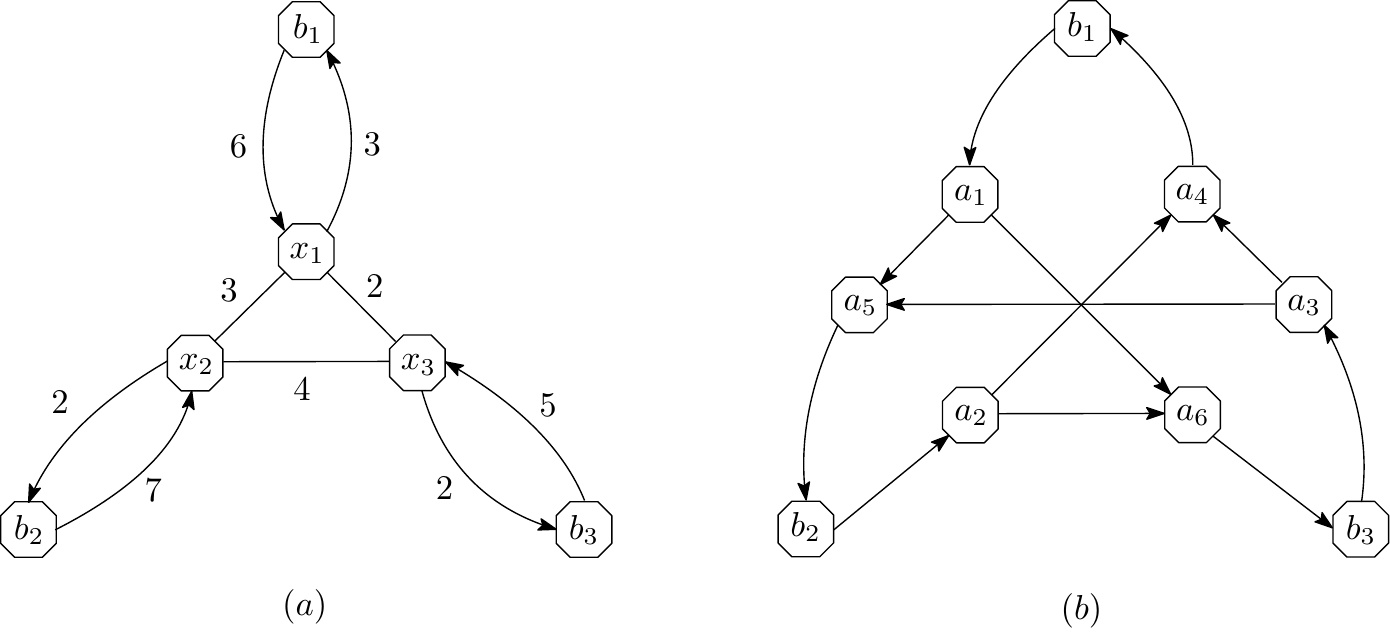}
	\caption{(a) An example of a network graph.  To avoid clutter, the edges
		with no specified direction are assumed to go in both directions
		with the same weight, (b) reconstructed network}
	\label{fig:Delta}
\end{figure}

We remark that such geometry is fairly realistic if we, for example,
consider the vertices $b_j$ to be internet service providers (ISPs)
who are eager to push the traffic addressed outside their network to
other ISPs as soon as possible. Still, the reconstruction does not
match the graph we had originally.

\begin{figure}[h]
	\centering
	\includegraphics[scale=0.9]{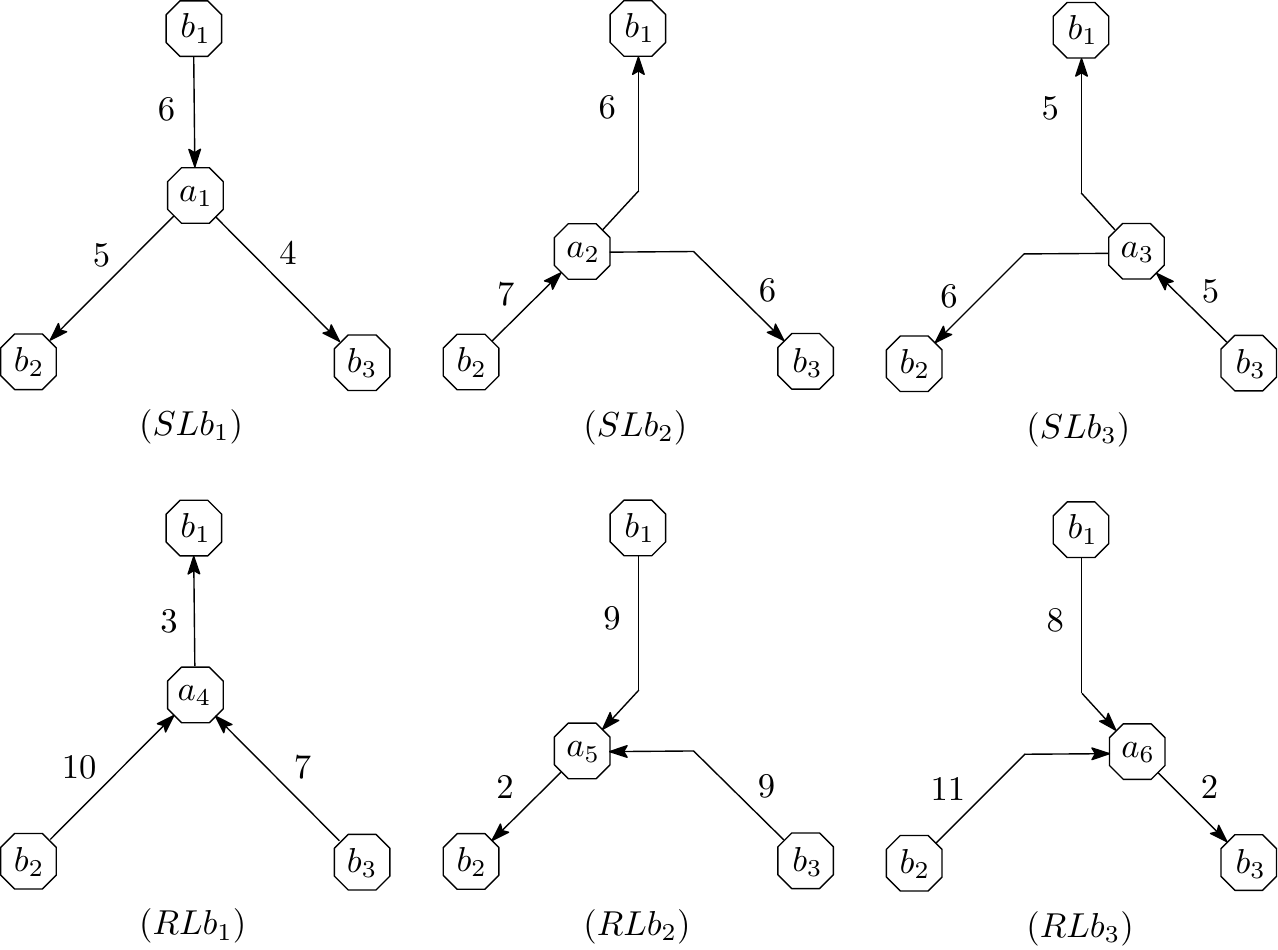}
	\caption{Representation of PCD on the network example through set of logical source and receiver trees.}
	\label{fig:DeltaReconst}
\end{figure}

However, the reconstruction \emph{is possible} if we know a priori
that the routes are symmetric, as they are in this example.  Then the
paths going out of a given source $b$ and the paths going to $b$
acting as the receiver have exactly the same topology.  This
additional information allows us to identify $a_1 = a_4$, $a_2=a_5$
and $a_3=a_6$ in Fig.~\ref{fig:Delta}(b) and thus recover the
original graph.  Since symmetric routing networks appear in
applications, we expect the optimized
Algorithm~\ref{algorithm:symmetric} to be a practical value.

\subsection{The effect of symmetric edge weights}

Symmetric network graph for which both the routing on the graph and
edge weights are symmetric can be considered as a special class of
general networks. For this class of networks, similar to the graphs
with symmetric routing (see eq.~\ref{eq:reversal}), same set of edges
are passed in the two reverse paths $\Path_1(u,v)$ and
$\Path_2(v,u)$. Additionally, the symmetric edge weights property
implies that for any vertex $x \in \Path_1$, then
$|\Path(v,x)| = |\Path(u,v)| - |\Path(u,x)|$. In other words, path
$\Path_2$ is fully identified by the available information form the
path $\Path_1$. This property of symmetric networks implies that
$\text{PCD}(v,w \succ u) = \text{PCD}(u \prec v,w)$, and as a result
the information from \textit{path correlation data} encoded in the form
of $(u \prec v,w)$-junctions can be applied to identify the
information for the $(v,w \succ u)$-junctions. From the
\textbf{reconstruction} point of view, same conditions are required
for exact reconstruction of graphs with symmetric edge weights as
discussed in previous sections and thereby
Algorithm~\ref{algorithm:symmetric} can be applied for reconstruction
purposes.


\section{Conclusion and Future Work}
\label{sec:conclusion_futurework}
In this paper, we solve the problem of reconstructing a network graph
from the measurements of PCD: the common length of any two paths
sharing an origin or a destination.  It has been recently demonstrated
that such measurements can be collected from the ambient traffic,
without encumbering the network with specialized probing packets.

We establish necessary and sufficient conditions for a network graph
to be reconstructible and describe the reconstruction algorithm.  When
the reconstruction is attempted on a graph violating those conditions,
the result of our algorithm is a minimal graph fitting the available
data.  More precisely, among the network graphs that have the same
PCD, the reconstructed graph has the minimal number of links, no
vertices of degree 2 and the minimal average degree among vertices of
degree 3 or more.  Our framework is general and no assumption is made
on the symmetricity of routing or edge weights of original graph is
required.  However, for graphs with symmetric routing less restrictive
conditions are shown to apply and a more specialized algorithm is
presented.

A possible direction for future work is to explore extending the class
of reconstructible graphs by making additional measurements.  The most
important obstacle to overcome is the non-separability condition.
Violation of this condition leads to multiple representations, in the
reconstructed graph, of a single vertex.  For example, for the network
shown in Fig.~\ref{fig:AltSource}, \ref{fig:AltSourceLogic} and
Fig.~\ref{fig:AltRoutingLogicalT2}, the application of the algorithm
will result a network in Fig.~\ref{fig:LastExamp}.  All internal
vertices except $u$ are recovered correctly, whereas the original
vertex $u$ separates into two representative vertices $u'$ and $u''$
in the reconstructed graph.  In future work we will explore whether
extending the weight model to include vertex weights may allow correct
identification of the separable vertices. The motivation for such
model is that load at a switch, e.g.\ due to network attacks, may
introduce performance degradation that affects not just the traffic
through a given link or router interface, but all the packets
traversing the switch.  Finally, as mentioned in the introduction, in
work to be reported elsewhere \cite{PrepUs2018}, we explain how the
methods of the current paper can be extended to Path Correlation Data
in which sampling and measurement noise leads to inconsistencies in
path weights reported for different paths.

\begin{figure}[H]
	\centering
	\includegraphics[scale=1]{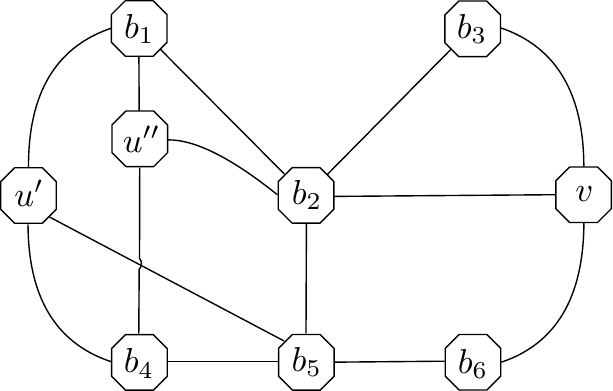}
	\caption{Separation of the internal vertex $u$ in the
          reconstructed network.}
	\label{fig:LastExamp}
\end{figure}


\bibliographystyle{plain}

\end{document}